\pdfoutput=1

\documentclass[preprint,1p]{elsarticle}
\journal{Journal of Computational Physics}
\usepackage{algorithmic,bm}
\usepackage{algorithm}
\usepackage{amssymb,url}
\usepackage{amsmath,amsfonts,mathrsfs}
\usepackage{subcaption}
\usepackage{color}
\usepackage[english]{babel}
\usepackage{cleveref}
\usepackage{graphicx}
\usepackage{epstopdf}

\newtheorem{lemma}{Lemma}

\newproof{proof}{Proof}

\newcommand{\Bm}{\mathfrak{L}_m^{\infty}}
\newcommand{\Bmi}{\mathfrak{L}_{m,i}^{\infty}}
\newcommand{\Bo}{\mathfrak{L}_o^{\infty}}
\newcommand{\Boi}{\mathfrak{L}_{o,i}^{\infty}}
\newcommand{\Em}{\mathscr{E}_m}
\newcommand{\Emi}{\mathscr{E}_{m,i}}
\newcommand{\Eo}{\mathscr{E}_{o}}
\newcommand{\Eoi}{\mathscr{E}_{o,i}}

\graphicspath{{./figures/}}

\begin{document}
\begin{frontmatter}
\title{Localised sequential state estimation for advection dominated flows with non-Gaussian uncertainty description}
\author[er]{Emanuele Ragnoli\corref{cor}}
\ead{emanuele.ragnoli@ie.ibm.com}

\author[mz]{Mykhaylo Zayats}
\ead{m.zayats1@nuigalway.ie}

\author[er]{Fearghal O'Donncha}
\ead{feardonn@ie.ibm.com}

\author[er]{Sergiy Zhuk}
\ead{sergiy.zhuk@ie.ibm.com}

\cortext[cor]{Corresponding author}
\address[mz]{ NUI Galway, Galway, Ireland}
\address[er]{IBM Research, Dublin, Ireland}

\begin{abstract}
This paper presents a new iterative state estimation algorithm for advection dominated flows with non-Gaussian uncertainty description of $L^\infty$-type: uncertain initial condition and model error are assumed to be pointvise bounded in space and time, and the observation noise has uncertain but bounded second moments. The algorithm approximates this $L^\infty$-type bounding set by a union of possibly overlapping ellipsoids, which are localized (in space) on a number of sub-domains. On each sub-domain the state of the original system is estimated by the standard $L^2$-type filter (e.g. Kalman/minimax filter) which uses Gaussian/ellipsoidal uncertainty description and observations (if any) which correspond to this sub-domain. The resulting local state estimates are stitched together by the iterative d-ADN Schwartz method to reconstruct the state of the original system. The efficacy of the proposed method is demonstrated with a set of numerical examples.
\end{abstract}

\begin{keyword}
data assimilation \sep
filtering\sep
minimax\sep
domain decomposition\sep
advection dominated flows
\end{keyword}
\end{frontmatter}

\section{Introduction}
\label{sec:introduction}
Consider an advection-diffusion process described by the following partial differential equation (PDE):
\begin{equation}\label{model}
  \begin{split}
u_t = -&\mu \cdot \nabla u+\epsilon \Delta u + f+e \quad \text{ in } \Omega\times (0,T)\\
&u(0,x)=u_0(x)+e_0(x),\quad u=0\, \quad \text{on } (0,T]\times\partial\Omega
\end{split}
\end{equation}
The initial state of the process, $u_0$ and the forcing term $f$ are presumed to be approximations of the ``true'' initial state and forcing respectively, and the error of this approximation is quantified by $e_0$ and $e$, uncertain parameters which are assumed to be just bounded ($L^\infty$-type uncertainty description): $|e_0(x)| \leq q_0(x)$ and $|e(t,x)|\le q(t,x)$ for given functions $q_0$ and $q$. In other words, every $e_0$ and $e$ satisfying the aforementioned inequality almost everywhere is equally possible.

The information about the dynamics of the state is obtained via a noisy observation process:
\begin{equation}\label{obs}
y(t,x) = Hu(t,x) + \eta,\quad  Hu(t,x) = \int_\Omega h(x-y)u(x,t)dx +\eta(t,x),
\end{equation}
where the noise $\eta(t,x)$ is of stochastic nature with zero mean and unknown but bounded second moments: $E[\eta^2(t,x)r(t,x)] \leq 1$ for a given $r$. Consider a filter, that is the accompanying process described by the following PDE:
\begin{equation}\label{filter}
  \begin{split}
\hat u_t &= -\mu \cdot \nabla \hat u+\epsilon \Delta \hat u + f + K(y-Hu) \text{in } \Omega\times (0,T)\\
u(0,x)&=u_0(x),\quad u=0\, \text{on } (0,T]\times\partial\Omega
\end{split}
\end{equation}
\textbf{The problem is} given $L^\infty$-type (non-Gaussian) uncertainty description, to design the gain $K$ so that the worst-case mean-squared estimation error, i.e. $\max_{e_0,e,\eta}E\|u-\hat u\|^2$ is minimal (in an appropriate norm).

In this work we solve the above problem by combining ideas from optimal control and numerical analysis. Specifically, the computational domain $\Omega$ is decomposed into a set of small non-overlapping subdomains, and, then, the $L^\infty$-constraints on $e$ and $e_0$ are approximated by $L^2$-type constraints, pretty much like circumscribing a rectangle by an ellipse of minimal volume. The error of approximating $L^\infty$-constraints by $L^2$-ellipsoid on a small sub-domain can be made quite small. This suggests to restrict~\cref{model} and~\cref{obs} to the introduced subdomains, and design a continous $L^2$-minimax filter for each subdomain. The aforementioned restriction is done by the adaptive Dirichlet-Neumann (ADN) domain decomposition (DD) approach since ~\cref{model} to accomodate the case of little or no diffusion. The resulting \emph{interconnected localised filters} are exchanging data with each other through boundary conditions: the continuity of the global state estimate across the subdomain interfaces is guaranteed by an alternating Schwartz approach.  Finally, the continuous filters are discretized in space by means of Finite Element Method (FEM), and a simplectic Runge-Kutta method is used for time integration. The resulting numerical algorithm, which approximates~\cref{filter} for the case of $L^\infty$-type model errors $e_0$ and $e$, and incomplete and noisy observations with random noise $\eta$ with uncertain but bounded second moments, is our main contribution.

\paragraph{Motivation and related work}
Problems like \cref{model}--\cref{filter} are fundamental in many fields including data assimilation for geophysical flows, and more specifically the study of ocean processes and events. Indeed, many marine based industries require accurate forecasts of the transport and trajectories of dissolved and suspended material. Examples include the transport of nutrients around aquaculture installations~\cite{odonncha2013physical}, forecasting oil spill evolution for remediation efforts~\cite{guo2009modeling} and monitoring releases from industrial operations~\cite{koziy1998three}, and \emph{data assimilation} is widely used to solve the aforementioned engineering problems. Data assimilation improves the accuracy of forecasts provided by physical models and evaluates their reliability by optimally combining \emph{a priori} knowledge encoded in equations of mathematical physics with \emph{a posteriori} information in the form of sensor data. Mathematically, many DA methods rely upon various approximations of stochastic filters. We refer the reader to~\cite{reich2015probabilistic,law2015data} for further discussions on mathematics behind data assimilation.

In the control/data assimilation literature, the problem of this paper is known as a filtering problem (if $\eta$ and $e_0$, $e$ are stochastic) or state estimation problem (for deterministic $\eta$, $e_0$, $e$). Theoretically, solution of the stochastic filtering problem for linear systems is given by the Kalman-Bucy filter~\cite{bensoussan1971filtrage}, provided $\eta$ and $e_0$, $e$ have appropriate (normal) distributions. In contrast, deterministic state estimators assume that errors have bounded energy and belong to a given bounding set. The state estimate is then defined as a minimax center of the reachability set, a set of all states of the physical model which are reachable from the given set of initial conditions and are compatible with observations. Dynamics of the minimax center is described by a minimax filter~\cite{kurzhanski_ellipsoidal_1997,nakonechny_minimax_1978, zhuk_kalman_2013, zhuk_minimax_2010}. In case of ellipsoidal bounding sets and linear dynamics, equations of the minimax filter coincide with those of Kalman-Bucy filter~\cite{krener1980kalman}.\\
In this paper we consider the case of deterministic $e_0$, $e$ and stochastic observation noise for practical reasons: indeed, a statistical description of the modelling errors/disturbances $e_0$, $e$ is often not available in many practical situations, e.g. in oceanography, but empirical estimates of the first and second moments of the measurements noise $\eta$ (e.g. pointvise bounds mentioned above) are usually provided by sensors~\cite{Lynn1973}. Since the classical Kalman/minimax filter cannot be applied directly for this ``hybrid'' uncertainty description, i.e. deterministic $e_0$, $e$ and stochastic $\eta$ with bounded second moments, on each subdomain we use the minimax filter for linear parabolic PDEs~\cref{model} proposed in~\cite{nakonechny_minimax_1978}. This latter filter is designed to work with stochatsic/determinstic uncertainties. We stress, however, that it does not apply directly to the case of $L^\infty$-type uncertainties considered here. A straightforward way to apply this filter in the considered case would be to approximate $L^\infty$-ellipsoid by $L^2$-ellipsoid which is very much like approximating a rectangle by the minimal ellipsoid which contains it. This approximation is quite crude, especially if the measure of the computational domain or/and the estimation horizon are large. In addition, the minimax filter is very demanding computationally and hence discretizing it over entire $\Omega$ does not scale well even in two spatial dimensions. However, as noted above, decomposing  the computational domain $\Omega$, and, then, approximating the $L^\infty$-constraints on each subdomain by $L^2$-type constraints does not introduce large errors, and, moreover, computing filters locally, on small subdomains becomes computationally tractable even for implicit time integrators, provided a proper domain decomposition approach has been chosen. Specifically, taking into account that the advective part in~\cref{model} is assumed to be dominant, we apply adaptive Dirichlet-Neumann (ADN) Domain Decomposition~\cite{gastaldi_adn_1998,quarteroni_domain_1999} which enforces boundary conditions across subdomain interfaces taking into account the direction of the advection. Note that implicit time integrators preserve dynamics of the state estimation error as it was outlined in~\cite{frank_symplectic_2014}, and hence our domain decomposition strategy combined with the simplectic Runge-Kutta method makes the numerical approximation of the estimation error computationally feasible and reliable. The latter is often not the case for state estimators based on explicit numerical methods.

This work is an extension of~\cite{ragnoli_localised_2015,ragnoli_domain_2014}. It is most related to the distributed Kalman/minimax filtering framework~\cite{mahmoud_distributed_2013} where, in contrast to the ideas of this paper, the ``distribution of filters'' is often done for a discrete model by decomposing a matrix, which represents a discretization of PDE's differential operator while here, we decompose the continuous problem, and discretize continuous  (in space and time) filters. The efficiency of \emph{interconnected localised filters} is demonstrated on a set of numerical examples. These experiments are characterised by idealised simulations of a concentration being transported either by a constant flow field or a non-stationary periodic flow filed. The benchmark for estimation is given by a correspondent known analytical solution, and a discussion of the computational complexity is included.

The rest of this paper is organised as follows: section \ref{sec:ProblemStatement} contains mathematical preliminaries; section \ref{sec:DDproblem} describes fully discrete interconnected localised filters; section \ref{sec:NumericalExperiment} presents the results of numerical experiments and discuss computational complexity; section \ref{sec:Conclusion} contains the conclusions and finally two appendixes complete the work with details of the FEM discretization and some proofs.

\section{Mathematical preliminaries}
\label{sec:ProblemStatement}
\textbf{Notation. } $\mathbb{R}^n$ denotes the $n$-dimensional Euclidean space. $\Omega$ a domain in $\mathbb{R}^n$, $\delta \Omega$ its boundary and $\Omega_T=(0,T]\times \Omega$ for some fixed time $T$. $\Omega_i$ is a subdomain of $\Omega$, and the intersection of the boundaries of a set of $\Omega_i$ is defined as the interface. $L^2([t_0,t_1], \mathbb{R}^n)$ denotes a space of square-integrable functions with values in $\mathbb{R}^n$. $H^1([t_0,t_1], \Omega)$ the Sobolev space of weak differentiable functions with support on $\Omega$ and $H^1_0(\Omega)$ is the space of functions in $H^1(\Omega)$ that vanish at the boundary. $L^\infty(0,T,H^1_0(\Omega))$  and $L^\infty(0,T,L^\infty(\Omega))$ are the spaces of almost everywhere bounded vector-functions with values in $H^1_0(\Omega)$ and $L^\infty(\Omega)$, respectively. $H^\star$ is the adjoint operator of $H$, $\delta(x-y)$ is the Dirac measure concentrated at $x$, $(\cdot,\cdot)$ is the canonical inner product in $\mathbb R^n$.

\textbf{State equation.}
Consider an advection diffusion problem described by the following linear parabolic equation:
\begin{equation}\label{eq:state}
\left\{ \begin{array}{ll}
u_t = Lu+f+e & \text{in } \Omega_T\\
u=u_0+e_0 & \text{on }  \{t=0\}\times \Omega\\
u=0\, & \text{on } (0,T]\times\partial\Omega\\
\end{array}\right.
\end{equation}
where $t$ and $x$ are the temporal and spatial variables, respectively, $\Omega$ is a bounded subset of $\mathbb{R}^n$ with Lipschitz boundary, $u_0,e_0 \in H^1_0(\Omega) $, $f,e\in L^2(0,T,L^\infty(\Omega))$, and $L$ is a uniformly parabolic~\cite[p.372]{Evans} differential operator. It is well known that in this case there exists a unique $u\in L^\infty(0,T,H^1_0(\Omega))$ verifying the equation~\eqref{eq:state} in the weak sense~\cite[p.372]{Evans}. To simplify the presentation, in what follows consider $L$ of the following form:
\begin{equation*}
\begin{array}{c}
Lu=-\mu \cdot \nabla u+\epsilon \Delta u\,,  \epsilon > 0
\end{array}
\end{equation*}
where $\mu\in C^1(\Omega_T)^n$ is a given divergence free vector field describing the flow transporting the quantity $u$. In what follows the case of advection-dominated flows, i.e. when the diffusion is strongly dominated by the advection (high Peclet number), will be considered. Note that the following results may be derived without major modifications for generic uniformly parabolic differential operators.

\textbf{Observation equation.} Assume that a function $y(t,x)$ is observed:
\begin{equation}\label{eq:obs}
y(t,x) = Hu(t,x)+\eta(t,x)\,,\quad Hu(t,x) = \int_\Omega h(x-y)u(x,t)dx\,,
\end{equation}
where $h$ is a given kernel function, and $\eta$ is a realization of a random field with zero mean and bounded and continuous (in $(t,x)$) correlation function. The function $y$ may be considered as measurements of the quantity $u$ subject to the measurement noise $\eta$, and $H$ is the mathematical model of the gauge.

\textbf{Uncertainty description.} Assume that $e_0$, $e$ and $\eta$ are uncertain parameters which represent error in the initial condition, model error (or an exogenous disturbance) and noise in the measurements. Further assume that $e_0$ and $e$ are elements of the given convex bounded set:
\begin{equation}\label{modelInfContEllipsoid1}
\Bm= \{ e_0(x),e(t,x): |e_0(x)| \leq q_0(x) , \; |e(t,x)|\le q(t,x) \}\,,
\end{equation}
where $q_0$ and $q$ are given weighting functions such that $0<\underline{q}_0\le q_0(x)\le \overline{q}_0<+\infty$ and $0<\underline{q}(t)\le q(t,x)\le \overline{q}(t)<+\infty$ for the given $\underline{q}_0,\overline{q}_0$ and $\underline{q},\overline{q}$. Note that $q_0$ and $q$ may be considered as design parameters which quantify our level of confidence in the initial condition and state equation: namely, $q_0$ may specify ``zones'' of $\Omega$ where the knowledge of the initial condition $u_0$ is more precise or less so, and $q$ defines zones of $\Omega$ where~\eqref{eq:state} holds almost exactly ($|e|\approx 0$ in that zone) or only up to a significant error ($|e|>0$) and these zones may vary over time. Statistically, this corresponds to the maximal entropy assumption, i.e., any $(e_0,e)\in\Bm$ have equal probability to appear in~\eqref{eq:state}. \\
In addition, assume that $\eta$ has bounded second moments (in $(t,x)$), that is:
\begin{equation}\label{observInfContEllipsoid1}
\Bo = \{\eta: E[\eta^2(t,x)r(t,x)] \leq 1 \}
\end{equation}
where $r$ is such that $0<\underline{r}(t)\le r(t,x)\le \overline{r}(t)<+\infty$ for given $\underline{r}$, $\overline{r}$. In fact, this assumption allows for an uncertainty in the statistical description of the observation noise $\eta$, which covers, in particular, a very practical case when the second moments of the observation noise are obtained from empirical estimators.

\textbf{The estimation problem} is to construct a computationally efficient estimate $\tilde u(T)$ of $u(T)$ with the minimal worst-case error in the direction $l\in L^2(\Omega)$, i.e., for any $l,v\in L^2(\Omega)$ the aim is to  search for a solution of the following problem:
\begin{equation}
\label{P1}
\begin{split}
  &\text{Find } \tilde u(T)\text{ such that: }\\
  &\sup_{(e_0,e)\in\Bm, \eta\in\Bo}E(l(\tilde u(T))-l(u(T)))^2\le \sup_{(e_0,e)\in\Bm, \eta\in\Bo}E(v(y)-l(u(T)))^2\,,\\
  &l(u)=\int_\Omega l(x)u(T,x) dx,\quad v(y)=\int_{\Omega_T} v(t,x)y(t,x)dxdt\,.
\end{split}
\end{equation}
In other words, a function $\tilde u(T)$ is constructed such that the worst-case mean-squared estimation error is minimal (see 2nd line in \eqref{P1}, provided that (i) $u$ solves the state equation~\eqref{eq:state}, and (ii) the error in the initial condition, $e_0$ and the model error $e$ are uncertain elements of the set $\Bm$, and (iii) the measurements noise $\eta$ belongs to the set $\Bo$.

It has been shown in~\cite{Zhuk2009d} that the optimal solution of the estimation problem \eqref{P1}, $\tilde u$ is the unique solution of an optimal control problem with a convex non-smooth cost functional (in the form of $L^1(\Omega)$-norm, the dual of $L^\infty$-norm) and a PDE constraint. To find the solution of this control problem one needs to solve Euler-Lagrange equations, which, in particular, implies that, to compute $\tilde u(t_2)$ for $t_2>T$ one needs to solve Euler-Lagrange equation for $t\in (0,t_2)$ as $\tilde u(t_2)$ cannot be expressed as a function of $\tilde u(T)$ and observations $y(t,x)$, $t\in(T,t_2]$. In other words, the estimate $\tilde u$ is not recursive. The reason for this is as follows: $\Bm$ is an ellipsoid of $L^\infty(\Omega)$ with respect to the $L^\infty(\Omega)$-norm, and the dual norm of the latter is given by the $L^1(\Omega)$-norm. Hence, $\Bm$ does not coincide with its dual set\footnote{This is obvious in the case of finite-dimensional Euclidean space where $\Bm$ would correspond to a rectangle and its dual will be a rhombus.}. On the other hand, the $L^2$-norm coincides with its dual norm and this property of the norm is necessary and sufficient to get estimates in the form of recursive filters, e.g. Kalman/minimax filter; see~\cite{Zhuk2009d}. Hence, a straightforward way to construct the recursive estimate $\hat u$ is to approximate the ellipsoid $\Bm$ by the $L^2(\Omega)$-ellipsoid:
\begin{equation}\label{modelContEllipsoid1}
\begin{split}
\Em = & \{ e_0(x),e(t,x): \int_{\Omega}e_0^2(x)q^{-2}_0(x)d\Omega + \\
+ & \int_{\Omega_T} e^2(t,x)q^{-2}(t,x) d\Omega dt \leq (T+1)\ell(\Omega) \}
\end{split}
\end{equation}
and $\Bo$ by
\begin{equation}\label{observContEllipsoid1}
\Eo = \{\eta: \int_{\Omega_T} E\eta^2(t,x)r(t,x)d\Omega dt \leq T\ell(\Omega) \}
\end{equation}
where $\ell(\Omega):=\int_{\Omega}d\Omega$ denotes the Lebesgue measure of the domain $\Omega$.

Intuitively, approximating $\Bm$ with $\Em$ is similar to approximating a rectangle by the minimal ellipsoid which contains it. Indeed, the level set of the $L^\infty$-type norm in the $n$-dimensional Euclidean space, i.e., $\{e=(e_1\dots e_n)^\top:\max_i|e_i|\le 1\}$, is a rectangle, and the level set of a $L^2$-type norm is an ellipsoid, i.e., $\{e=(e_1\dots e_n)^\top:\sum_{i=1}^ne^2_i\le 1\}$. Hence, it can be stated that $L^\infty$-type ($L^2$-type) norm has level sets of rectangular (ellipsoidal) shape for generic infinite-dimensional spaces. Consequently, as noted, $\Em$ can be considered as an ellipsoid of the space $L^2(\Omega)\times L^2(\Omega_T)$ containing $\Bm$. A similar argument can be applied to $\Bo$ and $\Eo$.\\

The key benefit of the aforementioned approximation is that the recursive estimate $\hat u$ of $u$ can be constructed, provided $(e_0,e)\in\Em$, $\eta\in\Eo$. Indeed, the estimate $\hat u$ of $u$ with minimal mean-squared estimation error, i.e., \[
\sup_{(e_0,e)\in\Em, \eta\in\Eo}E(l(\hat u(T))-l(u(T)))^2\le \sup_{(e_0,e)\in\Em, \eta\in\Eo}E(v(y)-l(u(T)))^2
\] admits the following representation: \[
l(\hat u(T)) = \int_{\Omega_T} r(t,x) (H p)(t,x)(y(t,x) - (Hw)(t,x)) dxdt + l(w(T))
\] provided $w$ solves
\begin{equation}\label{eq:w}
\left\{ \begin{array}{ll}
w_t = Lu+f& \text{in } \Omega_T\\
w=u_0 & \text{on }  \{t=0\}\times \Omega\\
w=0\, & \text{on } (0,T]\times\partial\Omega\\
\end{array}\right.
\end{equation}
and $p$ and $z$ solve the following Hamiltonian system of equations:
\begin{equation} \label{eq:dualSystem}
\left\{ \begin{array}{l}
    z_t = -L^\star z + H^\star rHp \text{ in }\Omega_T \\
    z(T,x) = l(x)  \text{ on } \Omega\\
    z(t,x) = 0 \text{ on } \partial\Omega\times [0,T]\\
    p_t = Lp + q^2z \text{ in }\Omega_T \\
    p(0,x) = q_0^2(x)z(0,x)  \text{ on } \Omega\\
    p(t,x) = 0 \text{ on } \partial\Omega\times [0,T]\\
\end{array}\right.
\end{equation}
The worst-case mean-squared estimation error is given by \[
\sup_{(e_0,e)\in\Em, \eta\in\Eo}E(l(\hat u(T))-l(u(T)))^2 = l(p)\,.
\]
Note that $\hat u$ can be represented as a filter, i.e., it can be shown that
\begin{equation} \label{eq:filter}
\left\{ \begin{array}{l}
    \hat u_t = L\hat u + f + PH^\star r(y-H\hat u) \text{ in }\Omega_T\\
    \hat u(t,x) = 0 \text{ on } \partial\Omega\times [0,T]\\
    \hat u(0,x)=u_0 \text{ on } \Omega
  \end{array}\right.
\end{equation}
where the operator $P$, a so called Riccati operator, is an integral operator of the following form:
\begin{equation}
\label{eq:contRicOp}
    (Pz)(t,x):=\int_\Omega k(t,x,\gamma) z(\gamma) d\gamma
\end{equation}
and $k$ is the kernel of the operator $P$, $k$ solves the following parabolic equation\footnote{$L_x k$ denotes the result of application of $L$ to $k$ w.r.t. variable $x$}:
\begin{equation}
\label{eq:operRicEq}
\begin{split}
  &\frac{\partial k}{\partial t} = L_x k + L_\gamma k + q^2(t,x)\delta(x-\gamma) - P(H^\star r (H k))\,, k(0,x,\gamma) = q_0^2(x)\delta(x-\gamma)\,, \\
& k(t,x,\gamma)=0\text{ for } (x,\gamma)\in\partial \Omega\times\partial\Omega\,.
\end{split}
\end{equation}
The estimate $\hat u$ defined by~\eqref{eq:filter} will be referred to as the \emph{minimax estimate} or \emph{minimax filter}. The worst-case mean-squared estimation error of the minimax estimate $\hat u$ is given by
\begin{equation}
  \label{eq:errorEst}
E(l(\hat u(T)-u(T))^2\le \sup_{(e_0,e)\in\Em, \eta\in\Eo}E(l(\hat u(T))-l(u(T)))^2 = \int_\Omega l(x)(Pl)(t,x)dx\,.
\end{equation}
Even though the minimax estimate $\hat u$ is optimal for the case of $L^2$-type uncertainties $\Em$ and $\Eo$, from the practical standpoint, the aforementioned approach of approximating the solution of \eqref{P1} by $\hat u$ has two major drawbacks:
\begin{itemize}
\item [A)] The approximation of $\Bm$ by $\Em$ is quite crude, especially if the measure of $\Omega$, $\ell(\Omega)$ or/and the final time $T$ are large, e.g., $\ell(\Omega),T>>1$.
\item [B)] Solving~\eqref{eq:filter} numerically, especially computing the Riccati operator $P$, becomes very expensive even for the case of two spatial dimensions.
\end{itemize}

\section{Localised interconnected filters}
\label{sec:DDproblem}
In order to address (A) above, namely, to provide a more accurate approximation of $\Bm$, $\Bo$, assume that $\Omega$ is split into a finite number of non-overlapping\footnote{By definition, $\Omega_1\in R^n$ and $\Omega_2\in R^n$ are non-overlapping if their intersection is of measure zero in $R^n$. } subsets $\Omega_i$ and define \[
\Bmi= \{ e_0(x),e(t,x): |e_0(x)| \leq q_0(x) , \; |e(t,x)|\le q(t,x), x\in\Omega_i \}\,.
\] It then follows that $(e_0,e)\in\Bm$ if and only if $(e_0,e)\in\Bmi$ for all $i$. In other words, a larger ``rectangle'' $\Bm$ equals to the union of smaller ``rectangles'' $\Bmi$ provided $\Omega=\cup\Omega_i$. The same holds true for $\Bo$. Hence, the aforementioned splitting does not ``increase the uncertainty''. In contrast, the $L^2$-ellipsoid $\Em$ does not possess such property simply because the union of ellipsoids is not an ellipsoid, generally speaking. 
Now, taking this representation into account, the following computational strategy is suggested:
\begin{itemize}
\item [1.] generate local problems by restricting the state equation, observation equation and $\Bm$, $\Bo$ to $\Omega_i$ and approximate the restrictions of $\Bm$, $\Bo$ to $\Omega_i$ by $\Emi$ and $\Eoi$ respectively;
\item [2.] employ an appropriate domain decomposition technique to ensure the continuity of the global solution, obtained by stitching together solutions of the local problems, across $\Omega$;
\item [3.] introduce the minimax filter for each local problem and discretize the local minimax filter by using FEM in space and midpoint/M\"{o}bius time integrator.
\end{itemize}
This computational strategy resolves (A) as the ``small'' ellipsoids $\Emi$ approximate the ``small'' rectangles $\Bmi$ and the union of the ellipsoids $\Emi$ is contained in the large ellipsoid~\eqref{modelContEllipsoid1} approximating the entire $\Bm$. Moreover, as noted, the large ``rectangle'' $\Bm$ equals to the union of smaller ``rectangles'' $\Bmi$ provided $\Omega=\cup\Omega_i$. In addition, (B) is also resolved since the computational cost of computing $P$ and $\hat u$ over a smaller domain $\Omega_i$ is reduced. The following section implements points 1.-3. In what follows the local minimax filters will be referred as local or localised filters.
The strategy of localisation that implements points 1.-3. is described in details in this section. More precisely, subsection~\ref{sec:filterDecomposition} shows how to restrict the state equation, the observation equation and $\Bm$, $\Bo$ to subdomain $\Omega_i$, and employ the iterative d-ADN Schwarz domain decomposition method; \ref{sec:local-minimax-filter} introduces the localised minimax estimate, subsection~\ref{sec:FEM} shows how to discretize the local problem by using the classical Finite Element Method (FEM) and how to discretize the local minimax filter; \ref{sec:reinitFilter} derives properties of the localised filters and finally \ref{sec:pseudo-observations} introduces the idea of the pseudo-observations and the localized strategy algorithm.

\subsection{Domain decomposition of the global problem} \label{sec:filterDecomposition}
Let the domain $\Omega$ be divided into $N$ non-overlapping domains $\Omega_1,\ldots, \Omega_N$ with $\Gamma_{i,j}=\partial \Omega_i \cap  \partial \Omega_j$ denoting the common boundary between them. $\Gamma = \cup_{ij}\Gamma_{ij}$ denotes their union (in the rest of this work referred to as the interface). In addition, the inflow and outflow parts of the $\Gamma_{i,j}$ and $\partial\Omega_i$  are defined below:
\[
\begin{split}
  \Gamma^{in}_{i,j} &= \{\bm{x} \in \Gamma_{i,j}: \mu(x) \cdot n(x) < 0\}\,,\\   \Gamma^{out}_{i,j} &= \{\bm{x} \in \Gamma_{i,j}: \mu(x) \cdot n(x) > 0\}\,,\\
\partial\Omega^{in}_{i} &= \{\bm{x} \in \partial\Omega_{i}: \mu(x) \cdot n(x) < 0\}\,,\\
\partial\Omega^{out}_{i} &= \{\bm{x} \in \partial\Omega_{i}: \mu(x) \cdot n(x) > 0\}\,,\\
\Gamma_i^{in} &=\{\Gamma_{i,j}^{in}:\Gamma_{i,j}^{in}\ne\varnothing\}\,,\\
\Gamma_i^{out} &=\{\Gamma_{i,j}^{out}:\Gamma_{i,j}^{out}\ne\varnothing\}\,.
\end{split}
\]
The continuous Global Problem~\eqref{eq:state} is approximated via a set of Local Problems referred to as the Decomposed Problem:
\begin{equation}\label{LocalProblemi}
  \left\{ \begin{array}{l}
   \frac{\partial u_i}{\partial t} = L_i u_i+f_i+e_i\\
   u_i(t,x) = 0 \text{,  on } \partial \Omega \cap \partial \Omega_i \\
   u_i(t,x) = u_j(t,x) \text{,  on } \Gamma^{in}_{i,j}\in \Gamma_i^{in} \\
   \frac{\partial u_i(t,x)}{\partial n} = \frac{\partial u_j(t,x)}{\partial n} \text{,  on } \Gamma^{out}_{i,j} \in \Gamma_i^{out}\\
   u_i(0,x) = u_{0,i}(x)+e_{0,i}(x)
  \end{array}\right.
\end{equation}
where the local operator $L_i$ is the restriction of the original operator $L$ on $\Omega_i$, and $f_i$, $e_i$, $u_{0,i}$ and $e_{0,i}$ are the restrictions of $f$, $e$, $u_{0}$ and $e_{0}$ onto $\Omega_i$, and $e_{0,i}$, $e_i$ belong to the restriction of $\Bm$ onto $\Omega_i$, namely
\begin{equation}\label{modelInfContEllipsoid_i}
\Bmi= \{ e_0(x),e(t,x): |e_0(x)| \leq q_0(x) , \; |e(t,x)|\le q(t,x) \text{ on } \Omega_i\}\,.
\end{equation}
In what follows, the problem~\eqref{LocalProblemi} will be referred to as the \emph{$i$-th local problem} and  $u_D$ is the solution of the Decomposed Problem if $u_D=u_i$ on $\Omega_i$. Clearly, the choice of the boundary conditions on the interface boundaries $\Gamma_{i}^{in}$ and $\Gamma_{i}^{out}$ guarantees the continuity of $u_D$ across the interface $\Gamma$. Boundary conditions on external boundaries $\partial \Omega \cap \partial \Omega_i$ are inherited from the global problem~\eqref{eq:state}. An obvious sufficient condition for the existence and uniqueness of a solution of the Decomposed Problem is proved in the following lemma:

\begin{lemma}\label{ExistenceUniqueness}
If $u_G$ is the unique solution of the Global Problem for some $e_0,e\in\Bm$ then it is the unique solution of the Decomposed Problem.
\end{lemma}
\begin{proof}
Take $e_0,e\in\Bm$  and assume that $u_{Gi}$ denotes the restriction of $u_G$ onto the subdomain $\Omega_i$. It is obvious that $u_G$ satisfies all boundary conditions over the interface $\Gamma$. Thus, it remains to show that $u_{Gi}$ solves the $i$-th Local Problem. Since $u_{Gi}|_\Gamma=u_{G}|_\Gamma$, where $\Gamma$ is the interface, it follows that $u_{Gi}$ solves the $i$-th Local Problem. The uniqueness is an obvious consequence.
\end{proof}

The restriction of the observation equation is obvious:
\begin{equation}\label{eq:obs:i}
y_i(t,x) = H_i u_i(t,x)+\eta_i(t,x)\,,\quad H_iu_i(t,x) = \int_{\Omega_i} h(x-y)u_i(x,t)dx\,,
\end{equation}
where \begin{equation}\label{eq:obsnoise:i}
\Bo = \{\eta: E[\eta^2(t,x)r(t,x)] \leq 1, \quad x\in\Omega_i \}\,.
\end{equation}

The Decomposed Problem described above is an application of a Domain Decomposition (DD) technique, namely the Adaptive Dirichlet Neumann method~\cite{gastaldi_adn_1998}. Since advection-dominated flows are considered, a further modification of the formulation~\eqref{LocalProblemi} is necessary. Indeed, for the pure advection problems the outflow boundary conditions on $\Gamma^{out}_{i,j}$ are not required as it follows from the physical properties of the flow $\mu$. This suggests to  incorporate the hyperbolic nature of the problem into~\eqref{LocalProblemi} by imposing the homogeneous Neumann condition in~\eqref{LocalProblemi}, which leads to a damped ADN (d-ADN) decomposition. The latter is known to work well for advection dominated problems~\cite{ciccoli_adaptive_1996}. The actual computational scheme is then carried out by solving for $u_i$ over $\Omega_i$ and iterating until convergence, a so called iterative Schwartz approach~\cite{quarteroni_domain_1999}: specifically, it starts with a set of initial solutions $\{u_i^0\}$, and compute $\{u_i^{n+1}\}$ from $\{u_i^{n}\}$, $n \geq 0$ by solving numerically the following problem:
\begin{equation}\label{LocalIterProblemi}
  \left\{ \begin{array}{l}
   \frac{\partial u_i^{n+1}}{\partial t} = L_i u_i^{n+1}+e_i+f_i\\
   u_i^{n+1}(t,x) = 0 \text{,  on } \partial \Omega^{in} \cap \partial \Omega^{in}_i \\
   u_i^{n+1}(t,x) = u_j^n(t,x) \text{,  on } \Gamma^{in}_{i,j}\in \Gamma_i^{in} \\
   \frac{\partial u_i^{n+1}(t,x)}{\partial n} = 0 \text{,  on } \Gamma^{out}_{i,j}\in \Gamma_i^{out} \\
   u_i^{n+1}(0,x) = u_{0,i}(x)+e_{0,i}(x)
  \end{array}\right.
\end{equation}
Informally, the purpose of the Schwartz iterations defined in~\eqref{LocalIterProblemi} is to enforce the continuity of the solution of the problem ~\eqref{LocalIterProblemi} along the interfaces. Once this is achieved, the iteration process can be stopped. While this work does not study the rate of the convergence of the iterative Schwartz d-ADN method, it is worth to mention that, to the best of our knowledge, no such result could be found in the literature. In this regard, note that if the direction of the flow is constant, only one iteration of the Schwartz method is required. In the general case, it can be shown that the sequence $\{u_i^{n}\}$ converges weakly in $H^1(\Omega_1) \times... \times H^1(\Omega_N)$ to the unique solution of the Decomposed Problem $u_D$, provided the latter exists~\cite{gastaldi_adn_1998}.

Finally, $\Bmi$ and $\Boi$ are approximated by $L^2$-ellipsoids. Specifically, to approximate $\Bmi$, $\frac{e_0^2(x)}{q_0^2(x)} \le 1$ and $\frac{e^2(t,x)}{q^2(t,x)}\le 1$ are integrated over $\Omega_T$ to obtain the approximating ellipsoid $\Emi$ of the following form:
\begin{equation}\label{modelContEllipsoid:i}
\begin{split}
\Emi = & \{ e_0(x),e(t,x): \int_{\Omega_i}e_0^2(x)q^{-2}_0(x)d\Omega_i + \\
+ & \int_{\Omega_i\times (0,T)} e^2(t,x)q^{-2}(t,x) d\Omega_i dt \leq (T+1)\ell(\Omega_i) \}
\end{split}
\end{equation}
which obviously contains $\Bmi$. It needs to be stressed that the union of the ``small'' ellipsoids $\Emi$, approximating $\Bmi$ is contained in the large ellipsoid~\eqref{modelContEllipsoid1} approximating the entire $\Bm$. \\
Similarly, $\Eo$ is approximated by:
\begin{equation}\label{observContEllipsoid:i}
\Eoi = \{\eta: \int_{\Omega_i\times (0,T)} E\eta^2(t,x)r(t,x)d\Omega dt \leq (1+T)\ell(\Omega_i) \}
\end{equation}
Note that the worst-case estimation error may be quite large if $T>>1$ or $\ell(\Omega_i)>>1$. This problem is resolved below, in section~\ref{sec:reinitFilter}.

\subsection{Interconnected localized minimax filters}
\label{sec:local-minimax-filter}
In this subsection, the minimax estimate $\hat u_i^{n+1}$ of $u_i^{n+1}$, the solution of the $n$-th Schwartz iteration for $i$-th Local problem, is introduced given $y_i$, $u_i^n$, and assuming that $e_{0,i},e_i\in\Bmi$, and $\eta_i\in\Boi$. Sometimes, $\hat u_i^{n+1}$ will be referred to as the $(n+1,i)$-filter.

Note that each local solution $u_i^{n+1}$ is the sum of a ``mean'' local solution $w_i^{n+1}$ and noisy part $q_i$, i.e., $u_i^{n+1} = w_i^{n+1}+q_i$, provided $w_i^{n+1}$ solves
\begin{equation}\label{eq:wi}
\left\{ \begin{array}{ll}
\frac{\partial w_i^{n+1}}{\partial t} = L_iw_i^{n+1}+f_i\\
w_i^{n+1}(t,x) = 0 \text{,  on } \partial \Omega^{in} \cap \partial \Omega^{in}_i \\
   w_i^{n+1}(t,x) = w_j^n(t,x)+q_j(t,x) \text{,  on } \Gamma^{in}_{i,j}\in \Gamma_i^{in} \\
   \frac{\partial w_i^{n+1}}{\partial n} = 0 \text{,  on } \Gamma^{out}_{i,j}\in \Gamma_i^{out} \\
   w_i^{n+1}(0,x) = u_{0,i}(x)
\end{array}\right.
\end{equation}
and $q$ solves
\begin{equation}\label{eq:qi}
\left\{ \begin{array}{ll}
\frac{\partial q_i}{\partial t} = L_iq_i+e_i\\
q_i(t,x) = 0 \text{,  on } \partial \Omega^{in} \cap \partial \Omega^{in}_i \\
   q_i(t,x) = 0 \text{,  on } \Gamma^{in}_{i,j}\in \Gamma_i^{in} \\
   \frac{\partial q_i}{\partial n} = 0 \text{,  on } \Gamma^{out}_{i,j}\in \Gamma_i^{out} \\
   q_i(0,x) = e_{0,i}(x)
\end{array}\right.
\end{equation}
Note that $w_i^{n+1}$ depends linearly on $q_j$ and $ w_j^n$, hence the minimax estimate of $w_i^{n+1}$ is given by $\hat w_i^{n+1}$, the solution of~\eqref{eq:wi} which corresponds to $w_i^{n+1} = \hat u_j^n$ on $\Gamma^{in}_{i,j}\in \Gamma_i^{in}$, where $\hat u_j^n$ denotes the $(n,j)$-filter obtained on the $n$-th iteration of the Schwartz iterative procedure. Since $u_i^{n+1} = w_i^{n+1}+q_i$, it follows that $y_i = H_iu_i^{n+1}+\eta_i = H_iw_i^{n+1}+Hq_i+\eta_i$. Hence, the noisy part $q^{n+1}_i$  can be estimated from the shifted local measurements $\tilde{y}_i:=y_i-H_i\hat w_i^{n+1}$. It should be stressed that, although the noisy part $q^{n+1}_i$ is independent of the corresponding noisy parts $q_j^{n+1}$, its minimax estimate does depend on observations $\tilde{y}_i $ which, in turn, depend on $\hat w_j^n$, so $\hat q_j^{n+1}$ changes over the course of the Schwartz iterative procedure. That said, the minimax estimate $\hat u_i^{n+1}$ can be computed as the sum of $\hat w_i^{n+1}$ and $\hat q_i^{n+1}$, i.e.,
\begin{equation}
  \label{eq:uwq}
l_i(\hat u_i^{n+1}) = l_i(\hat w_i^{n+1}) + l_i(\hat q_i^{n+1})\,,
\end{equation}
where, analogously to Section~\ref{sec:ProblemStatement}, the minimax estimate $\hat q_i^{n+1}$ is represented as follows:
\begin{equation}
  \label{eq:MF_qi}
l(\hat q_i^{n+1}(T)) = \int_{\Omega_i\times (0,T)} \frac{r_i(t,x)}{(T+1)\ell(\Omega_i) } (H_i p_i)(t,x)\tilde y_i(t,x) dxdt
\end{equation}
provided $p_i$ and $z_i$ solve the following Hamiltonian system of equations:
\begin{equation} \label{eq:pzi}
\left\{ \begin{array}{l}
    \frac{\partial z_{i}}{\partial t} = -L^\star_i z_i + \frac{H_i^\star r_iH_ip_i}{(T+1)\ell(\Omega_i)} \\
    z_i(T,x) = l_i(x)  \text{ on } \Omega_i\\
    z_i(t,x) = 0 \text{ on } \partial \Omega^{in} \cap \partial \Omega^{in}_i\\
    z_i(t,x) = 0 \text{ on } \Gamma^{in}_{i,j}\in \Gamma_i^{in}\\
    \frac{\partial z_i}{\partial n} = 0 \text{,  on } \Gamma^{out}_{i,j}\in \Gamma_i^{out} \\
    \frac{\partial p_i}{\partial t} = L_ip_i + (T+1)\ell(\Omega_i) q_i^2z_i \\
    p_i(0,x) = (T+1)\ell(\Omega_i) q_{0,i}^2(x)z_i(0,x)  \text{ on } \Omega_i\\
    p_i(t,x) = 0 \text{ on } \partial \Omega^{in} \cap \partial \Omega^{in}_i\\
    p_i(t,x) = 0 \text{ on } \Gamma^{in}_{i,j}\in \Gamma_i^{in}\\
    \frac{\partial p_i}{\partial n} = 0 \text{,  on } \Gamma^{out}_{i,j}\in \Gamma_i^{out} \\
\end{array}\right.
\end{equation}
Here $l_i$ stands for the restriction of $l$ onto $\Omega_i$. The local worst-case mean-squared estimation error is given by \[
\sup_{(e_0,e)\in\Emi, \eta\in\Eoi}E(l_i(\hat u_i^{n+1}(T))-l_i(u_i(T)))^2 = l_i(p_i)\,.
\]
In the following section the $(n+1,i)$-filter will be discretized (in space) by using FEM.

\subsection{Finite Element Approximation for the  $(n+1,i)$-filter}
\label{sec:FEM}

Finite Element Method consists of (i) reformulating the problem~\eqref{LocalIterProblemi} in the weak form, and (ii) applying the Galerkin projection method to construct $\bm{u}_i^{n+1}=(u^{n+1}_{i1}(t)\dots u^{n+1}_{i N_{nd}^i}(t))^\top$, the FEM approximation of the solution $u_i^{n+1}$ in the so called FEM space: \[
u_i^{n+1}=\sum_{k=1}^{N_{nd}^i} u^{n+1}_{ik}(t)\phi_k+O(\frac{1}{(N_{nd}^i)^2})\text{ in } L^2(\Omega_i)\,.
\] provided $u_i^{n+1}\in H^2(\Omega)$. An important feature of the FEM approximation $\bm{u}_i^{n+1}$ is that it converges in $L^2(\Omega)$ but the coefficient $ u^{n+1}_{ik}(t)$ approximates $u_i^{n+1}(x_k,t)$, the value of $u_i^{n+1}$ at the FEM node $x_k$, provided $u_i^{n+1}$ is continuous in space. Here $\{x_s\}_{s=1}^{N_{nd}^i}$ represents a so-called FEM grid. The reader is referred to section~\ref{sec:appendix_FEM} where the detailed derivation of the FEM discretization for $i$-th local subproblem is provided for the case of two spatial dimensions. In what follows the FEM representation of the minimax estimate is derived and that converges to the continuous estimate provided the dimension of the FEM subspace, $N_{nd}^i$, increases.

The following notations are introduced: $\bm{u}_i^0$ is the FEM approximation of the restriction of $u_0$ onto $\Omega_i$, $\bm{l}_i$ is the FEM approximation of $l_i$, $\bm{M}_i$ is the local mass matrix, $\bm{S}_i$ is the local stiffness matrix (see~\eqref{locStiff}), $\bm{f}_i(t;\bm{\hat u}_j^n)$ is the local source vector (see~\eqref{locSource}), $\bm{\hat u}_j^n$ is the FEM approximation of $\hat u_j^n$. Moreover,  define
\begin{equation}
  \label{eq:FEM_matrices}
  \begin{split}
    \bm{C}_i&:=\{h(x_n-z_m)\}_{n,m=1}^{N_{nd}^i}\,,\quad \bm{R}_i=\operatorname{diag}(r_i(x_1)\dots r_i(x_{N_{nd}^i}))\,,\\
    \bm{Q}_i(t)&:=\operatorname{diag}(q_i^2(x_1,t)\dots q_i^2(x_{N_{nd}^i},t))\,,\quad
\bm{Q}_{0,i}(t):= \operatorname{diag}(q_{0,i}^2(x_1)\dots q_{0,i}^2(x_{N_{nd}^i}))\,,\\
\bm{y}_i &= (y_i(x_1,t),\dots, y_i(x_{N_{nd}^i}))^\top\,, \quad \gamma_{T,i}:=(T+1)\ell(\Omega_i)\,.
  \end{split}
\end{equation}
The following lemma provides the FEM approximation for the $(n+1,i)$-filter and its estimation error.
\begin{lemma}\label{l:MF_FEM}
The continuous minimax estimate $\hat u_i^{n+1}$ can be approximated as follows: for any $l_i\in L^2(\Omega_i)$ it holds
\begin{align}
&l_i(\hat u_i^{n+1})= (\bm{l}_i, \bm{\hat u}_i^{n+1}) +O(\frac{1}{(N_{nd}^i)^2})\,,   \label{eq:integral_est}\\
&\sup_{(e_0,e)\in\Emi, \eta\in\Eoi}E(l_i(\hat u_i^{n+1}(T))-l_i(u_i^{n+1}(T)))^2 = l_i(p_i) = (\bm{l}_i,\bm{P}_i(T)\bm{l}_i) +O(\frac{1}{(N_{nd}^i)^2})   \label{eq:integral_est_err}\\
\end{align}
where $\bm{\hat u}_i^{n+1}$ and $\bm{P}_i$ solve the following ODE:
\begin{equation}
\label{eq:filteri_FEM}
\begin{split}
  &\dfrac{d \bm{\hat u}_i^{n+1}}{dt} = \bm{S}_i \bm{M}_i^{-1}\bm{\hat u}_i^{n+1}
 + \gamma_{T,i}^{-1}\bm{P}_i \bm{C}_i^\top\bm{R}_i^{\frac 12}\bm{M}_i\bm{R}_i^{\frac 12}(\bm{y}_i - \bm{C}_i\bm{\hat u}_i^{n+1})
+ \bm{f}_i(t;\bm{\hat u}_j^n)\,,\\
& \dfrac{d\bm{P}_i}{dt} = \bm{S}_i \bm{M}_i^{-1} \bm{P}_i + \bm{P}_i \bm{M}_i^{-1}\bm{S}^\top_i + \gamma_{T,i}\bm{Q}_i^{\frac12}\bm{M}_i\bm{Q}_i^{\frac12}- \gamma_{T,i}^{-1}\bm{P}_i\bm{C}_i^\top\bm{R}_i^{\frac 12}\bm{M}_i\bm{R}_i^{\frac 12}\bm{C}_i\bm{P}_i\,,\\
&\bm{P}_i(0) = \gamma_{T,i}\bm{Q}_{0,i}^{\frac 12}\bm{M}_i\bm{Q}_{0,i}^{\frac 12}\,,\quad \bm{\hat u}_i^{n+1}(0)=\bf{u}_i^0\,.
\end{split}
\end{equation}
\end{lemma}
Equation~\eqref{eq:filteri_FEM} represents the FEM approximation of the $(n+1,i)$-filter. It has two ``correctors'': the first one steers the $(n+1,i)$-filter towards the observed data, and the second one, $\bm{f}_i(t;\bm{\hat u}_j^n)$ enforces the continuity across the interfaces between the subdomains. The proof of the lemma is given in the appendix right after the detailed description of the FEM discretization.


\subsubsection{Pointwise estimates}
\label{sec:pointwise-estimates}
It is stressed that~\eqref{eq:integral_est_err} and~\eqref{eq:integral_est} provide integral estimates as $l_i(\hat u_i^{n+1}) =\int_{\Omega_i} l_i(x) u_i^{n+1}(x,T)dx$. \\
Indeed the estimate of $(\bm{l}_i, \bm{M}_i\bm{u}_i^{n+1}(T))$, the discrete version of $l_i(\hat u_i^{n+1})$, is given by $(\bm{l}_i, \bm{\hat u}_i^{n+1}(T))$ so that, in fact, $\bm{\hat u}_i^{n+1}(T)$ provides and estimate of $\bm{M}_i\bm{u}_i^{n+1}(T)$, the vector of projections of $ u_i^{n+1}$ onto the FEM subspace  $L:=\operatorname{lin}(\{\phi_s\})$: $\bm{M}_i\bm{u}_i^{n+1}(T) = (\langle u_i^{n+1},\phi_1\rangle_{L^2(\Omega_i)}\dots \langle u_i^{n+1},\phi_{N_{nd}^i}\rangle_{L^2(\Omega_i)})^\top$.  It turns out that, thanks to the properties of the FEM approximation, one can employ the estimate of $l_i(\hat u_i^{n+1})$ to get an estimate of $u_i^{n+1}(x_s,T)$. Indeed, $\bm{u}_i^{n+1} = (u_i^{n+1}(t,x_1)\dots u_i^{n+1}(t,x_{N_{nd}^i}))^\top$, and so, as noted above, $\bm{M}_i^{-1}\bm{\hat u}_i^{n+1}(T)$ provides the estimate of $\bm{u}_i^{n+1}(T)$. More specifically, the $s$-th component of $\bm{M}_i^{-1}\bm{\hat u}_i^{n+1}$ provides an estimate of $u_i^{n+1}(t,x_s)$.


The estimation error of the aforementioned pointwise estimate is computed here. The straightforward approach, i.e., to use with $\bm{l}_i:=\bm{M}_i^{-1}\bm{l}_i^s$ with $\bm{l}_i^s=(0\dots 1\dots0)^\top$ does not provide a meaningful estimate as in this case \[
(\bm{l}_i^s, \bm{u}_i^{n+1}-\bm{M}_i^{-1}\bm{\hat u}_i^{n+1})\le (\bm{l}_i^s,\bm{M}_i^{-1}\bm{P}_i(T) \bm{M}_i^{-1} \bm{l}_i^s)^\frac 12
\]
and $(\bm{l}_i^s,\bm{M}_i^{-1}\bm{P}_i(T) \bm{M}_i^{-1} \bm{l}_i^s)$, the $s$-th element on the diagonal of the Riccati matrix $\bm{P}_i(T)$ grows unbounded. Indeed, since $(\bm{l}_i,\bm{P}_i(T)\bm{l}_i) $ approaches $l_i(p_i)$ when the dimension of the FEM subspace, $N_{nd}^i$ increases, and components of vector $\bm l_{ik}=l_i(x_k)$ does not depend on $N_{nd}^i$, it follows that the components of the matrix $\bm{P}_i$ must decay. On the other hand, $\bm{P}_i\bm{l}_i = \bm{d}_i = \bm{M}_i\bm{p}_i$ and so $\bm{M}_i^{-1}\bm{P}_i(T)\bm{l}_i = \bm{p}_i(T)$ and $\bm{p}_i(T)$ approaches $p_i(T)$ when $N_{nd}^i$ increases. Hence, the components of $\bm{M}_i^{-1}\bm{P}_i(T)$ are bounded for any $N_{nd}^i$. As a result, $\bm{M}_i^{-1}\bm{P}_i(T) \bm{M}_i^{-1} $ grows unbounded together with $\bm{M}_i^{-1}$ when $N_{nd}^i$ increases. Note that $\bm{M}_i^{-1}\bm{l}_i^s$ grows unbounded for any $s$ as it represents the ``FEM approximation'' of the Dirac measure $\delta(x-x_s)$ which has infinite $L^2(\Omega_i)$ norm. When the dimension of the FEM subspace increases, $\bm{M}_i^{-1}\bm{l}_i^s$ gets closer and closer to $\delta(x-x_s)$ (in the weak sense), and thus its $L^2$-norm grows. To overcome this, one should use a different error estimate, namely
\begin{equation}
  \label{eq:ME_i_FEM}
E(\bm{l}_i^s, \bm{u}_i^{n+1}-\bm{M}_i^{-1}\bm{\hat u}_i^{n+1})\le (\bm{l}_i^s,\bm{P}_i(T) \bm{M}_i^{-1}\bm{l}_i^s)^\frac 12\,.
\end{equation}
The rationale behind this is as follows: as noted above, $\bm{M}_i^{-1}\bm{P}_i(T)\bm{l}_i = \bm{p}_i(T)$ and $\bm{p}_i(T)$ approaches $p_i(T)$. Hence, the components of $\bm{M}_i^{-1}\bm{P}_i(T)$ are bounded for any $N_{nd}^i$. Even though one cannot derive~\eqref{eq:ME_i_FEM} directly as the proposed framework is optimal for the integral estimates like~\eqref{eq:integral_est_err} and~\eqref{eq:integral_est},  the validity of~\eqref{eq:ME_i_FEM} is confirmed by the numerical experiments (see Figure~\ref{e2_ricFEMInc}).

\subsection{$(n+1,i)$-filter with reinitialisation}
\label{sec:reinitFilter}

It easy to check that the minimax estimate $\bm{\hat u}_i^{n+1}$ is invariant with respect to the uniform rescaling of the ellipsoids $\Eoi$ and $\Emi$. Indeed, by examining~\eqref{eq:filteri_FEM} it is easy to find that multiplying $\bm{P}_i$ by a positive constant $\alpha$ is the same as dividing $\bm{Q}_{0,i}$, $\bm{Q}_i$ and $\bm{R}_i$ by this same $\alpha$ which implies the aforementioned invariance. This observation is used to further mitigate the error of approximating $\Bmi$, $\Boi$ by $\Emi$ and $\Eoi$. As it follows from the equation for $\bm{P}_i$ in~\eqref{eq:filteri_FEM}, the matrices $\bm{Q}_{0,i}^{-1}$, $\bm{Q}_i^{-1}$ and $\bm{R}_i$ 
are multiplied by the same constant, $\gamma_{T,i}^{-1}=\frac{1}{(T+1) \ell(\Omega_i)}$. It should be stressed that, for large $T>>1$ or large subdomains with $\ell(\Omega_i) >>1$, the error of approximating $\Bmi$, $\Boi$ by $\Emi$ and $\Eoi$ might become critical (see Figure~\ref{ric_gl_loc}): indeed, as it follows from~\eqref{eq:integral_est_err}, larger Riccati matrix $\bm{P}_i$ corresponds to larger estimation error; on the other hand, small $\frac{1}{(T+1) \ell(\Omega_i)}$ neutralize the impact of the quadratic term in the Riccati equation and amplifies the contribution of the source term. Hence, it is particularly important to keep the factor $\frac{1}{(T+1) \ell(\Omega_i)}$ as close as possible to $1$. To this end, one needs to design the domain decomposition of $\Omega$ so that $\ell(\Omega_i) \le 1$. In addition, thanks to the Markovian property of $\bm{\hat u}_i^{n+1}$, the size of the estimation horizon $T$ can be taken as small as required. Indeed, $\Bmi$ and $\Boi$ are uniform both in time and space, and therefore a decomposition technique may be applied in time. Namely, assuming that $\ell(\Omega_i) \le 1$ one can take any $T:=\varepsilon>0$, compute $\bm{\hat u}_i^{n+1}$ over $(0,\varepsilon)$ by using the recipe of lemma~\ref{l:MF_FEM}, and then
computing the estimate for $(k \varepsilon, (k+1)\varepsilon)$, dividing $\bm{Q}_i^{-1}$ and $\bm{R}_i$ by $1+\varepsilon\approx 1$ and starting the Riccati equation from $(1+\varepsilon)\bm{P}_i(k\varepsilon)$ in order to compute the estimate for the next window $((k+1) \varepsilon, (k+2)\varepsilon)$. It turns out that the proposed reinitialisation procedure allows to drastically reduce the impact of the error of approximating $\Bmi$, $\Boi$ by $\Emi$ and $\Eoi$ (see Figure~\ref{ric_gl_loc}).

\subsection{Pseudo-observations}
\label{sec:pseudo-observations}

It should be noted that the interconnections between the local filters $\hat{\bm{u}}^{n+1}_i$ are implemented by means of the source terms $\bm{f}_i(t;\hat{\bm{u}}_j^n)$: as a result the information from the interface (1D set in our case) is spread around in the domain and affects the nodes of the local estimate $\hat{\bm{u}}^{n+1}_i$ which are not necessarily close to the aforementioned interface. This, in turn, allows to push the information brought by observations $\textbf{y}_i$ on the domain $\Omega_i$ to the internal FEM nodes of the adjacent domains. The algorithm for computing $\hat{\bm{u}}^{n+1}_i$ is summarized in~\eqref{localisedalgorithm}.

On the other hand, the impact of observations on a local estimate depends on the structure of the local observation matrix $\bm{C}_i$. Specifically, if the observations $y(x,t)$ are localized at a specific region (e.g., $h$ has compact support within a subdomain of $\Omega$) of the global domain $\Omega$, it is possible that $h$ vanishes over a number of subdomains $\Omega_i$. In this case $\bm{C}_i=0$. This, in turn, may impact the uncertainty propagation associated with the local filters. Indeed, as it follows from~\eqref{eq:filteri_FEM}, the so-called innovation term $\gamma_{T,i}^{-1}\bm{P}_i \bm{C}_i^\top\bm{R}_i^{\frac 12}\bm{M}_i\bm{R}_i^{\frac 12}(\bm{y}_i - \bm{C}_i\bm{\hat u}_i^{n+1})$
disappears, provided $\bm{C}_i=0$. In this case, the impact of model errors from $\Omega_i$ is, in fact, neglected as the proposed procedure cannot communicate the corresponding information to the Riccati matrices on the adjacent subdomains. In this case, the local estimation error represented by means of the discrete Riccati operator $\bm{P}_i\bm{M}_i^{-1}$ may be underestimated.

A possible solution used in this work is to introduce "pseudo" observations: namely, the Dirichlet data that comes from the adjacent subdomains can be treated as "pseudo" observations. In this way, the impact of the model errors on adjacent domains can impact the estimate $\hat{\bm{u}}^{n+1}_i$. However, it is stressed that the Riccati equation is not affected even in this case. The reader is referred to the following section for numerical assessment of the proposed localised filtering strategy.

\begin{algorithm}
\caption{Algorithm of localised minimax filter method}
\label{localisedalgorithm}
\begin{algorithmic}
\REQUIRE
\STATE $T \;\;\;\;\;\;\;\;\;\;\;\;\;\;\;\;\;\;\;\;\;\;\;\;\;\;\;\;\;$ // number of time steps
\STATE $globalproblem \;\;\;\;\;\;\;\;\;$  // description of global physical problem
\STATE $errorlevel \;\;\;\;\;\;\;\;\;\;\;\;\;\;\;$     // acceptable level of Schwartz iteration error
\STATE GetInterfaceError() // computes the difference between estimates on the interface
\STATE $\;\;\;\;\;\;\;\;\;\;\;\;\;\;\;\;\;\;\;\;\;\;\;\;\;\;\;\;\;\;\;$ // nodes obtained from adjacent subdomains
\STATE $subproblems$ = DecomposeProblem($globalproblem$)
\FOR {$t=1$ \TO $T$}
  \FOR {$subdomain$ \textbf{in} $subdomains$}
    \STATE DiscretizeSubproblemByFem($subproblem$, t)
    \STATE UpdateBoundaryData($subproblem$, $subproblems$, t)
    \IF {HasObservations($subproblem$)}
	  \STATE InitObservations($subpoblem$, $t$)
	\ELSE
	  \STATE InitPseudoObservations($subpoblem$, $t$)
	\ENDIF
    \STATE SolveRiccatiEquation($subproblem$, $t$)
    \STATE SolveFilterEquation($subproblem$, $t$)
  \ENDFOR
  \STATE
  \STATE $error$ = GetInterfaceError($subproblems$, $t$)
  \WHILE {$error > errorlevel$}
    \FOR {$subdomain$ \textbf{in} $subdomains$}
      \STATE UpdateBoundaryData($subproblem$, $subproblems$, $t$)
      \STATE SolveFilterEquation($subproblem$, $t$)
    \ENDFOR
    \STATE $error$ = GetInterfaceError($subproblems$, $t$)
  \ENDWHILE
\ENDFOR
\end{algorithmic}
\end{algorithm}

\section{Numerical Experiments} \label{sec:NumericalExperiment}
The efficacy of the interconnected minimax filters is illustrated here with a set of numerical examples. First, a discrete in time representation of~\eqref{eq:filteri_FEM} is constructed. Note that the matrix Differential Riccati Equation (DRE) for $\bm{P}_i$ in~\eqref{eq:filteri_FEM} requires non-standard numerical integration techniques: for example, a standard explicit Runge Kutta (RK) method fails to integrate through the singularities~\cite{schiff_natural_1999}. One way to overcome this issue is to apply the M\"{o}bius Transformation that maps the DRE into its Hamiltonian representation, that can be effectively solved by symplectic midpoint method with reinitialisation at each time step~\cite{frank_symplectic_2014}. Following~\cite{frank_symplectic_2014}  the discrete in time system of linear Hamiltonian equations is introduced:
\begin{equation}
\label{eq:riccatiMP}
\left(
  \begin{matrix}
    \bm{U}_{k+1}\\\bm{V}_{k+1}
  \end{matrix}
\right)
= 2\left(
    \begin{matrix}
      I-\frac{h}2 \bm{M}_i^{-1}\bm{S}_{i,k+0.5} && \frac{h}2 \bm{B}_i\\
      \frac{h}2 \bm{D}_{i}  && I-\frac{h}2 (\bm{M}_i^{-1}\bm{S}_{i,k+0.5})^T
    \end{matrix}
\right)^{-1}
\left(
  \begin{matrix}
    \bm{P}_{i,k}\\I
  \end{matrix}
\right)-
\left(
  \begin{matrix}
    \bm{P}_{i,k}\\I
  \end{matrix}
\right),
\end{equation}
where \[
\begin{split}
    \bm{D}_i &= \bm{M}_i\bm{C}_{i} \bm{R}_{i,k+0.5}^{\frac 12}\bm{M}_i\bm{R}_{i,k+0.5}^{\frac 12} \bm{C}_{i}\bm{M}_i,\\
    \bm{B}_i &= \bm{M}^{-1}_i\bm{Q}^{\frac 12}_{i,k+0.5}\bm{M}_i\bm{Q}^{\frac 12}_{i,k+0.5}\bm{M}^{-1}_i.
\end{split}
\]
Here, subscript $k$ denotes the index of the points of the uniform time discretization with the step $h$. Subscript $k+0.5$ means that the corresponding matrix or vector is evaluated in the midle of the time interval $[t_k,t_{k+1}]$.

The $i$-th local Riccati matrix is found as $\bm{P}_{i,k+1}=\bm{U}_{k+1}\bm{V}_{k+1}^{-1}$ for $k>0$ and $\bm{P}_{i,0}=\bm{Q}_{0,i}^{-1}$. The aforementioned Hamiltonian system is then solved by using the symplectic midpoint method for the following reason: it was pointed out in~\cite{ZhukSISC13} that the time discretization of the filter equation and DRE must preserve quadratic invariants, e.g., non-stationary Lyapunov functions, which motivates one to apply the symplectic midpoint method, a symplectic implicit RK-method of second order.
This said, the equation for $\bm{u}_i^{n+1}$ (see~\eqref{eq:filteri_FEM}) is discretised as follows:
\begin{equation}
\label{eq:feedbackMP}
\begin{split}
  \hat{\bm{u}}^{n+1}_{i,k+1} & = -\hat{\bm{u}}_{i,k}^{n+1}
  + (I-\bm{M}_i^{-1}\bm{S}_{i,k+0.5} + G\bm{M}_i\bm{C}_{i})^{-1} \\
  & \times \left[2\hat{\bm{u}}_{i,k}^{n+1} + \bm{M}_i^{-1}\bm{f}_{i,k+0.5}(\hat{\bm{u}}_{j,k+0.5}^n) \right.\\
  & + G \left. \left(\bm{y}_{i,k+0.5}(t) - \frac 12 \bm{M}_i\bm{C}_{i}\hat{\bm{u}}_{i,k}^{n+1} \right)
  \right], \\
  \hat{\bm{u}}_{i,0}^{n+1} & = \bm{u}_{0,i}\,
\end{split}
\end{equation}
where \[
G = \frac12 \left(\bm{P}_{i,k} + \bm{P}_{i,k+1} \right)
  \bm{M}_i\bm{C}_{i} \bm{R}_{i}^{\frac 12}\bm{M}_i\bm{R}_{i,k+0.5}^{\frac 12}.
\]

The fully discrete interconnected localised minimax filters~\eqref{eq:feedbackMP}-~\eqref{eq:riccatiMP} are then iterated according to the Algorithm 1 in order to obtain the estimate of a solution of the linear advection dominated equation in two spatial dimensions in a set of two idealised experiments: one with a stationary flow field and another one with a non-stationary periodic flow field. In both experiments  the localised filters are compared  against the ground-truth and, in the second experiment, the localised filters are also compared to the global (non-decomposed) minimax filter, i.e., the standard minimax filter which approximates $\Bm$ and $\Bo$ by $\Em$ and $\Eo$, and does not use domain decomposition and reinitialization. This latter comparison illustrates the following points:
\begin{itemize}
\item $L^2$ non-decomposed filter does overestimate uncertainty which makes it of little or no use in practise,
\item interconnected localised minimax filters provide quite accurate uncertainty estimates in the considered examples,
\item drastic reduction of the computational cost in the case of localised filters.
\end{itemize}
\subsection{Experiment 1}
\begin{figure}
    \centering
    \includegraphics[width=\linewidth, height=18mm]{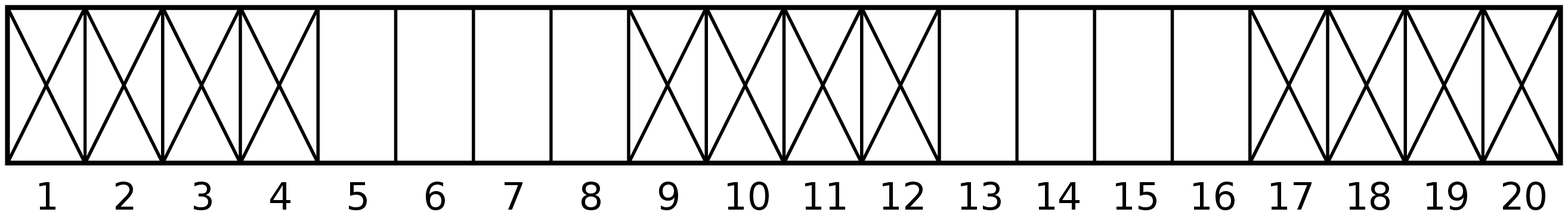}
    \caption{Configuration of Domain Decomposition and sensor locations for the first experiment.}
    \label{e1_conf}
\end{figure}

\begin{figure}
\centering
\begin{minipage}{0.49\textwidth}
  \includegraphics[width=\linewidth]{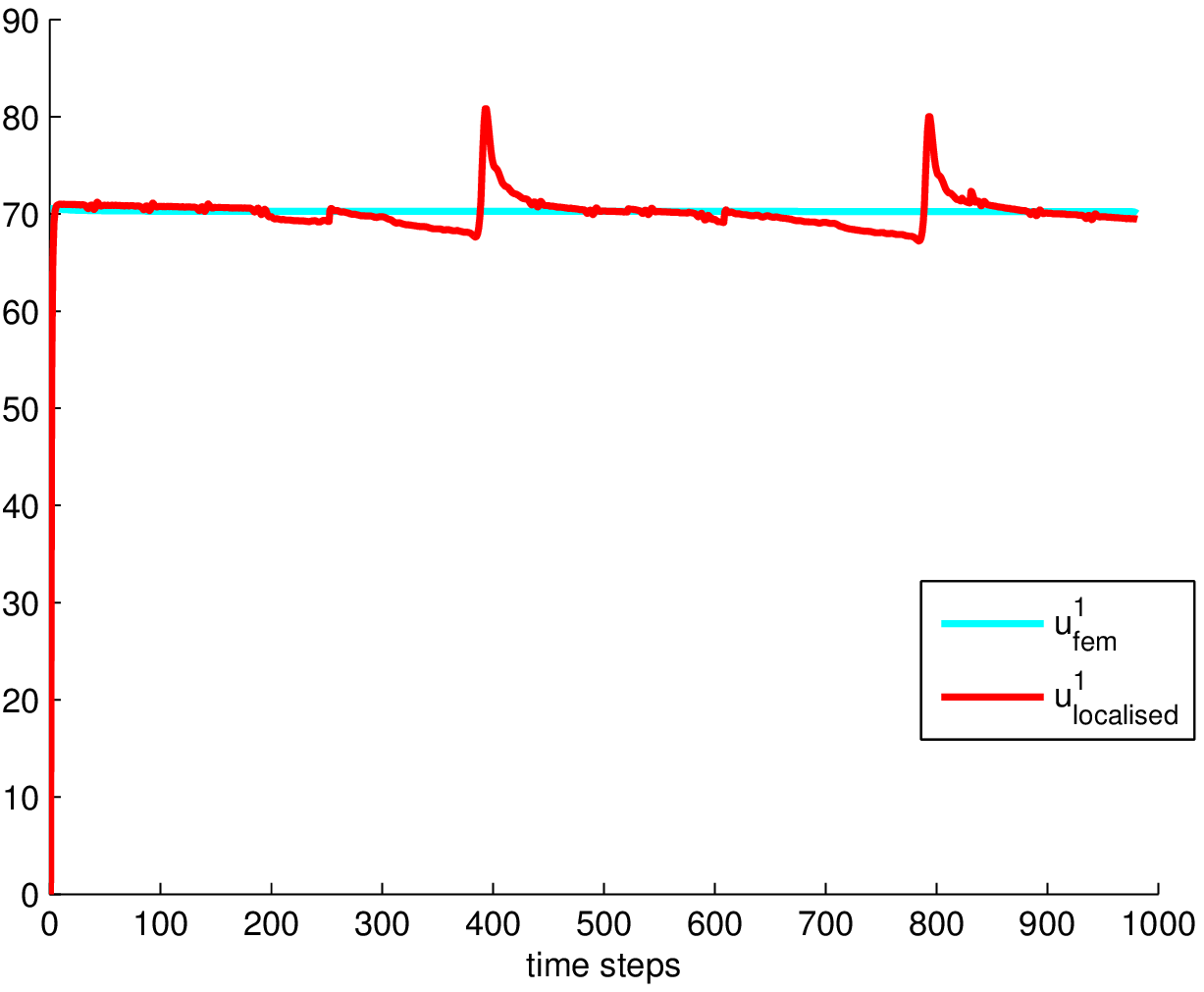}
  \caption{Spatial norm of the localised filter estimate $u_{localised}$ and mono domain FEM solution $u_{fem}$ (ground-truth spatial norm $\approx70$). }
  \label{e1_sol_norm}
\end{minipage}
\begin{minipage}{0.49\textwidth}
  \includegraphics[width=\linewidth]{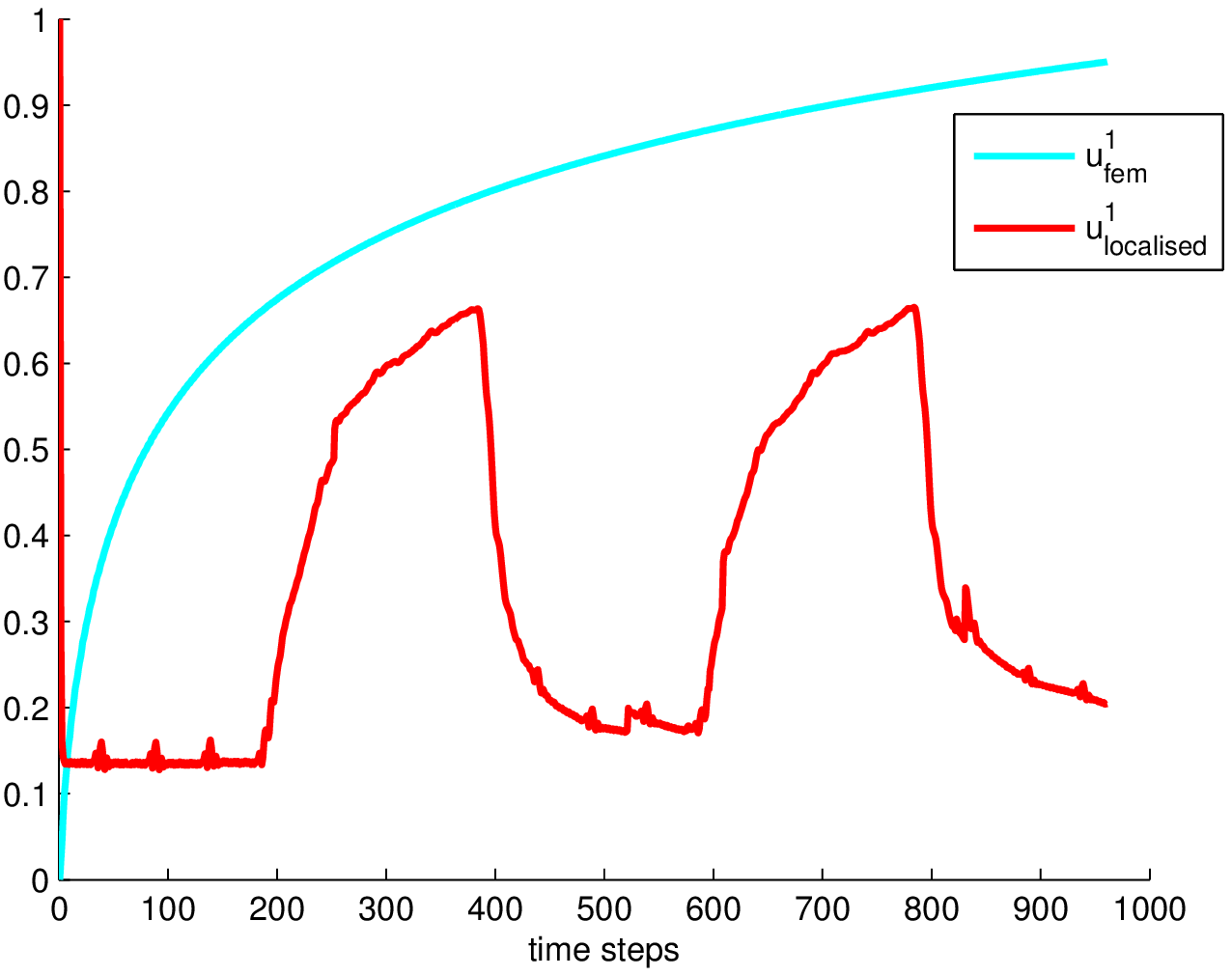}
  \caption{Spatial error of the localised filter estimate and mono domain FEM solution.}
  \label{e1_sol_er}
\end{minipage}
\end{figure}

\textbf{FEM discretization.} In this experiment a two dimensional rectangular domain of the size $[0,1]\times[0,20]$ is further discretized by $4500$ bilinear finite elements. DD is applied by decomposing the domain into 20 subdomains over the x-axis of the equal size (see Figure~\ref{e1_conf}) $[0,1]\times[0,1]$ and discretized by $225$ finite elements each. The underlying flow field is defined by the constant vector-function $\mu=[0.2; 0]$ and the constant diffusion coefficient $\epsilon=10^{-5}$. The timestep is taken to be $0.1$ and the length of the simulation is set to be $1000$ time steps allowing the concentration to completely transition from the right to the left of the domain. Note that the resulting FEM model is quite imprecise in that it quickly diverges from the analytical solution which is available in this case. This has been made intentionally in order to illustrate that the localised filters can improve the quality of the numerical solutions by using observed data and without knowing the initial conditions.

\textbf{Observations.} Define the following two-dimensional Gaussian function:
\begin{equation}
    \label{GaussianFunc}
    u_a(x,y,t) = \frac{1}{\sigma^2 2\pi} e^{-\frac{1}{2} \left(\frac{x-x_0-m_x}{\sigma}\right)^2} e^{-\frac{1}{2}\left(\frac{y-y_0-m_y}{\sigma}\right)^2}
\end{equation}
where $\sigma$, $m_x$ and $m_y$ are the diffusion and advection of the initial concentration $u_a(x,y,0)$; $x_0$ and $y_0$ define its center. Let $u^1_a(x,y,t)$ be the function as in~\eqref{GaussianFunc} with parameters $\sigma = 0.06+2t\epsilon = 0.06 + 0.02t$, $m_x = t*\mu_x = 200t$, $m_y = t*m_y = 0$, $x_0 = 0.25$, $y_0 = 0.25$. It is not difficult to check that the function $u^1_a(x,y,t)$ satisfies the original advection-diffusion equation~\eqref{eq:state} with the idealised flow field $\mu$ defined as above. In what follows, it serves as a ground-truth, and, in particular, the observations are sampled by restricting $u^1_a(x,y,t)$ onto the nodes in subdomains $\Omega_i$, $i\in I_{obs}:=\{1\dots4, 9\dots12, 17\dots20\}$ (see Figure~\ref{e1_conf}). This is achieved by setting $h(y-x)=\delta(y-x)$ for $y\in\Omega_i$, and $0$ for $y\not\in\Omega_i$, $i\in I_{obs}$. As a result, the observation matrix $\bm C_i$ consists of the rows of the inverted mass matrix $\bm M_i$ if the corresponding FEM node is observed and $0$ row otherwise. In fact, the matrix product $\bm M_i \bm C_i$ is a diagonal with its components equal to $1$ for the observed nodes and $0$ otherwise. The observations are corrupted by the observation noise with values uniformly distributed within the interval $[-0.5;0.5]$. The statistical characterization of this noise is given below.

\textbf{Uncertainty description.} The ellipsoids are chosen as defined by the functions $q_0$, $q$ and $r$, constant in time and space. Hence, $\bm{Q}^{\frac 12}_{i,k+0.5}\bm{M}_i\bm{Q}^{\frac 12}_{i,k+0.5} = \bm{Q}_i \bm{M}_i$ and the matrix $\bm{Q}_i=0.1I$ where $I$ is the identity matrix.  This choice reflects the low trust in the FEM model and, in a similar fashion, the absence of initial conditions is taken into account by defining $\bm{Q}_{0,i}=0.1I$.  The weighting matrix $\bm{R}_i$ is also diagonal: $\bm{R}_i=\operatorname{diag}(R_{i,1}\dots R_{i,N_{nd}^i})$. This means that the second moments of the observation noise, $\eta_i$ are required to verify the following inequality:
\begin{equation}
\label{eq:noiseDesc}
    \sum_{j=1}^{N_{nd}^i} R_{i,j}E(\eta_i^j)^2\le \Delta t \ell(\Omega_i)
\end{equation}
Here the Lebesgue measure of the subdomain $\ell(\Omega_i)=1$ and $\Delta t = 0.1$ is the size of the reinitialisation interval. Furthermore, $R_{i,j}=12$, i.e., the reciprocal of the variance of the $[-0.5;0.5]$-uniformly distributed random variable. It should be noted that our model of the observations noise is a robust version of the conventional statistical noise description, i.e., a realisation of any random variable $\eta_i$, which satisfies~\eqref{eq:noiseDesc}, could, in principle, ``corrupt'' the ``true concentration''. As a result, the proposed estimator is robust with respect to errors in second moment approximations, and the matrix $\bm{R}_i$ quantifies the magnitude of the moment approximation errors: roughly speaking, large/small $R_{i,j}$ restricts/loosens the admissible set of $\eta_i$.

Consequently, the estimate generated by Algorithm 1, $u_{\text{localised}}$, is compared against $u^1_a$ and $u_{fem}$ by applying the following error metrics:
\begin{itemize}
\item Spatial norm: $n_s(u)(t)= \|u\|$
\item Spatial error: $e_s(u)(t)= \frac{\|u - u_a\|}{\|u_a\|}$
\item Estimation error: $e_e(u)=\frac{\int_{0}^{T}\|u(t) - u_a(t)\| dt}{\int_{0}^{T}\|u_a\|dt}$
\end{itemize}

In Figures~\ref{e1_sol_norm} and \ref{e1_sol_er} the spatial norm and the spatial error of the localised filters are compared against the non-decomposed (mono-domain) FEM solution $u_{fem}$ of the problem with the exact initial condition ${u^1}_{a}(x,y,0)$. Figure~\ref{e1_sol_norm} shows that the spatial norm of the ground-truth is estimated correctly by $u_{fem}$. The localised filters $u_{localised}$ tend to estimate the norm correctly as well. The spikes in the graph happen when the spill enters a subdomain equipped with sensors (subdomains 10-12 and 17-20). Figure \ref{e1_sol_er} shows that, as it was expected, the $u^1_{fem}$ quickly diverges from the ground-truth due to the high model error and quite large time step, in contrast to $u^1_{localised}$ which start to diverge only when the concentration leaves the subdomains with sensors. The latter is due to the fact that the observation operator is zero over those subdomains  (subdomains 5-8 and 13-16) and the filters are driven by the erroneous FEM model only. The respective estimation errors are $e_e(u^1_{\text{fem}})=78\%$ and  $e_e(u^1_{\text{localised}})=39\%$.

\subsection{Experiment 2}
\begin{figure}[htpb]
  \subcaptionbox{Observations: $dt=25$, rel. err. 84.4\%.\label{e2_obs25}}%
{\includegraphics[width=6cm,keepaspectratio]{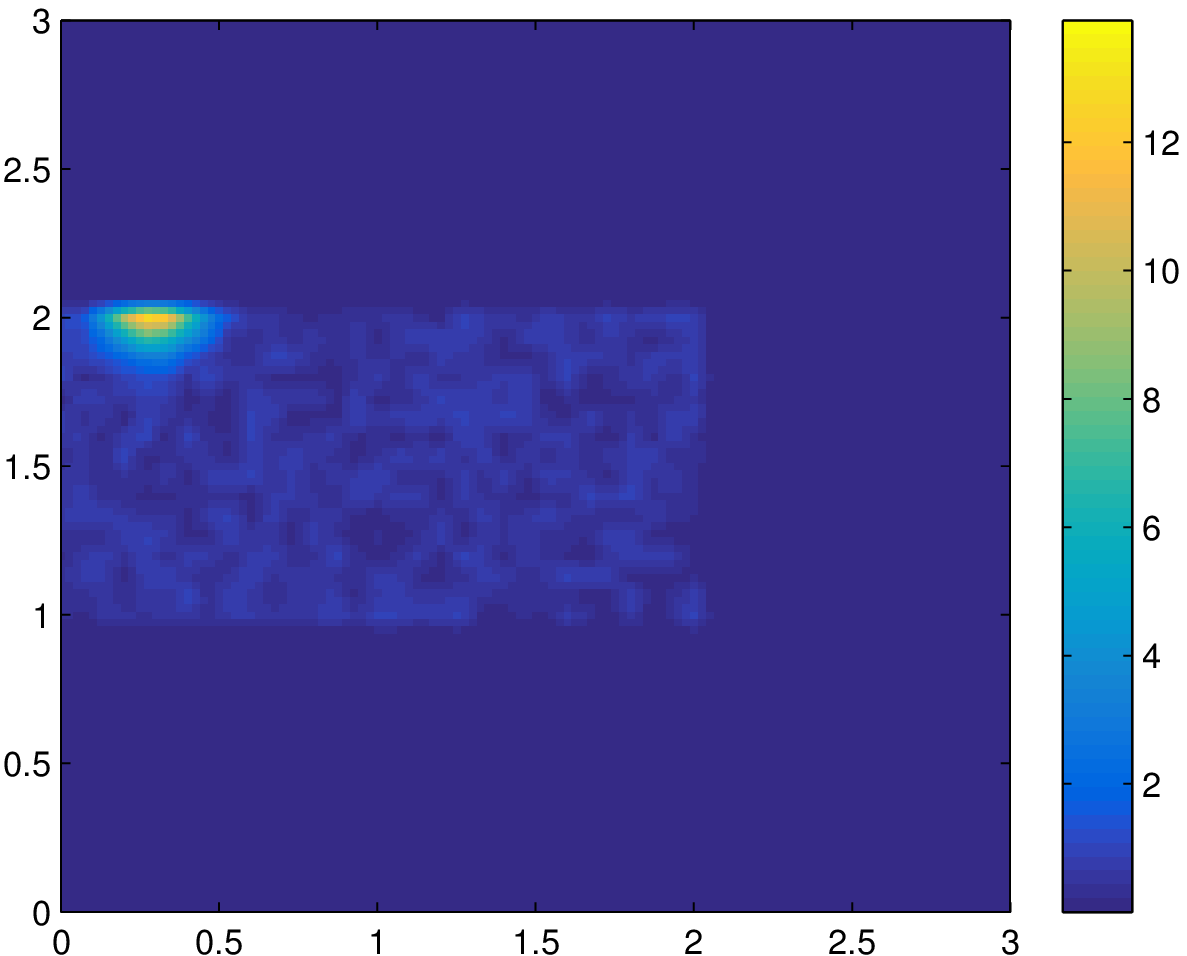}}\hfill
\subcaptionbox{Loc.estimate: $dt=25$, rel. err. 10.8\%.\label{e2_est25}}%
{\includegraphics[width=6cm,keepaspectratio]{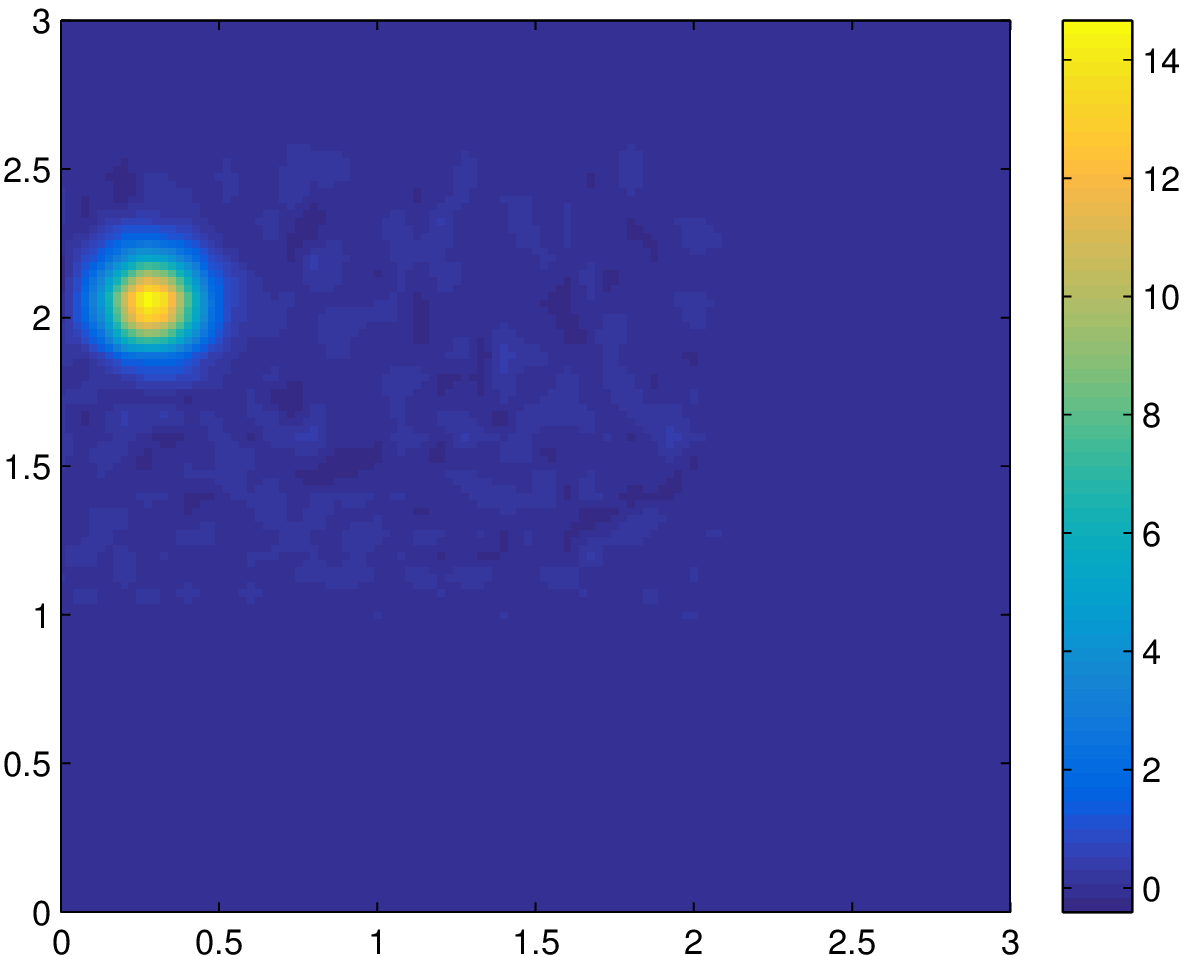}}

\subcaptionbox{Observations: $dt=180$, rel. err. 68.3\%. \label{e2_obs180}}%
{\includegraphics[width=6cm,keepaspectratio]{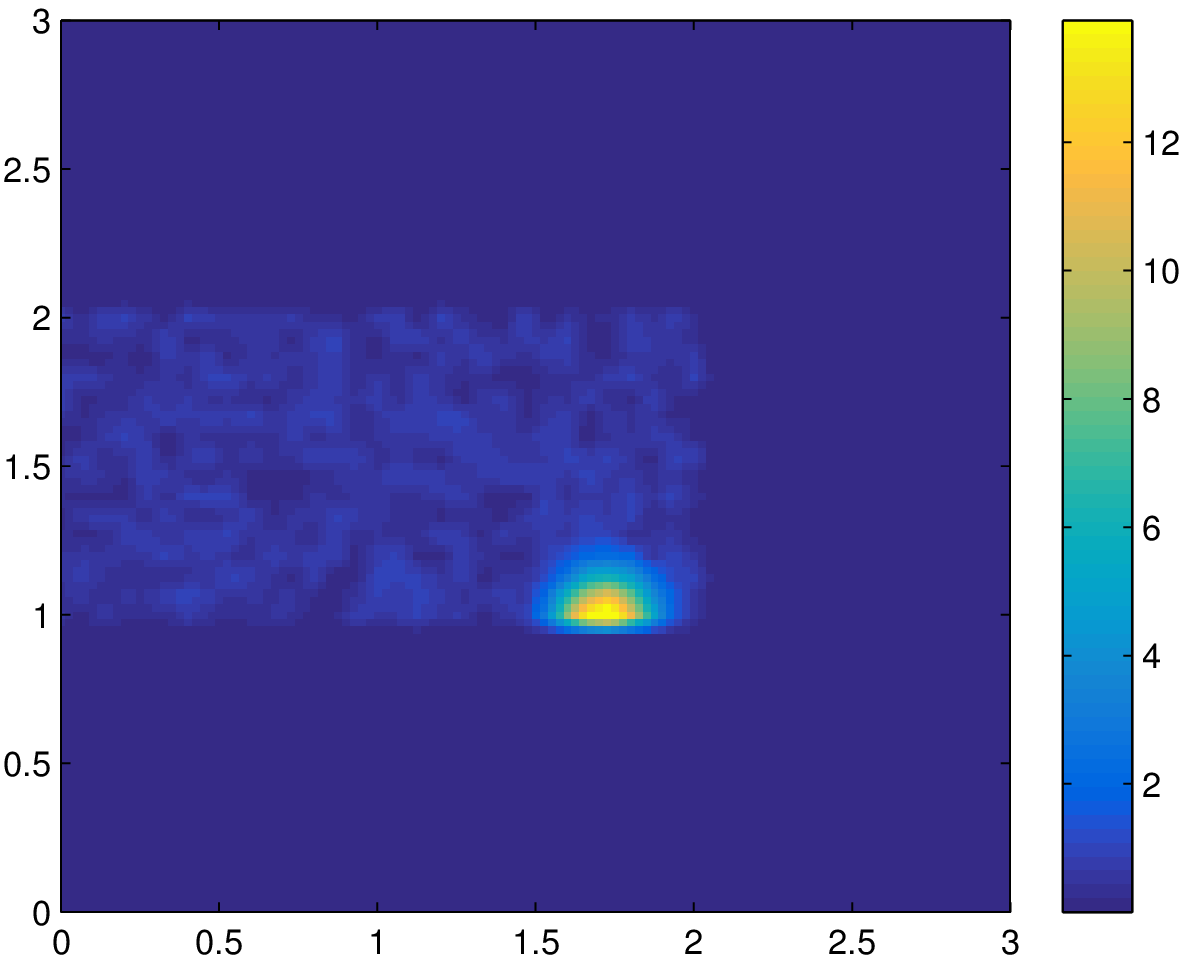}}\hfill
\subcaptionbox{Loc.estimate: $dt=180$, rel. err. 10.3\%. \label{e2_est180}}%
{\includegraphics[width=6cm,keepaspectratio]{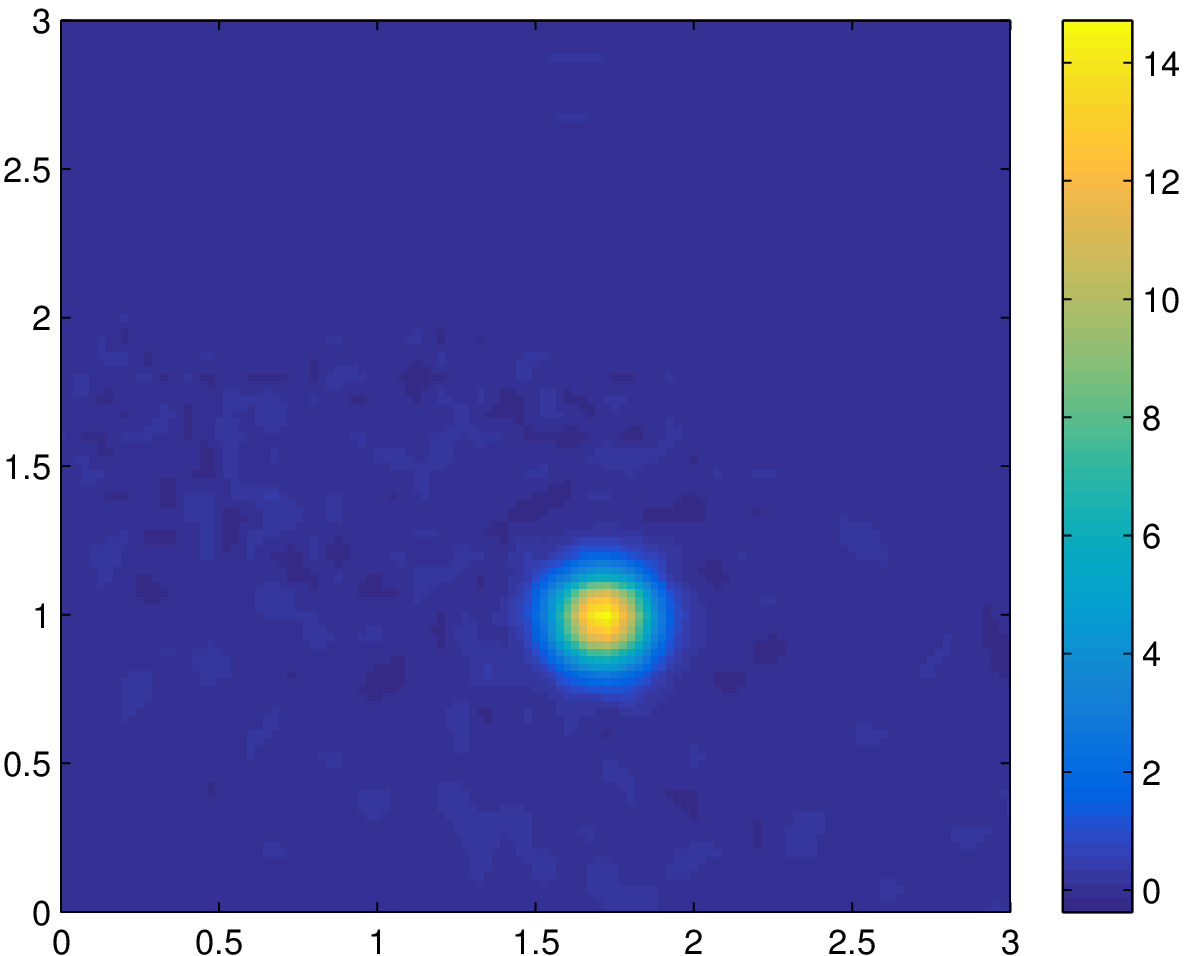}}

\subcaptionbox{Config. of Experiment 2.\label{e2_conf}}%
{\includegraphics[width=5cm,keepaspectratio]{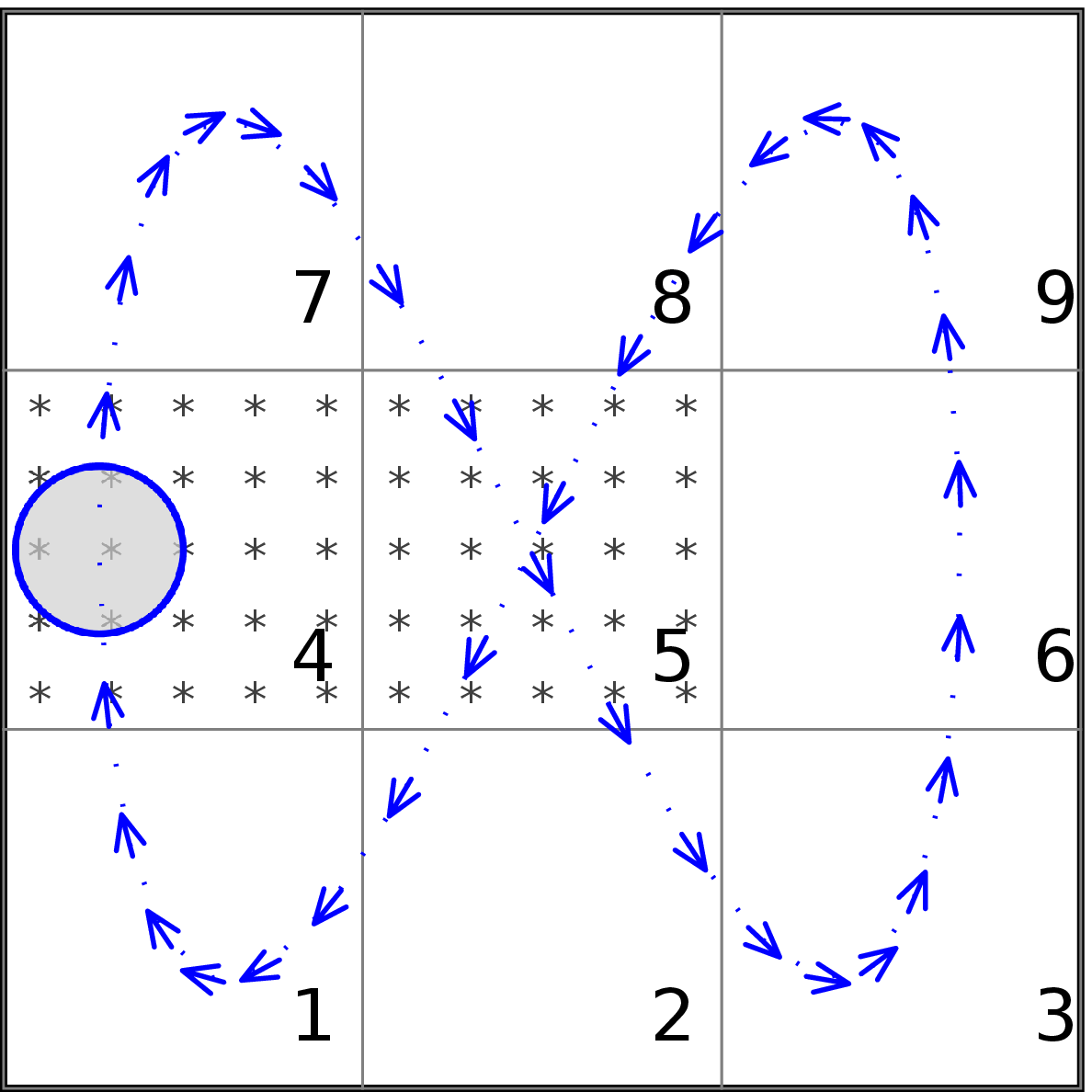}}\hfill
\subcaptionbox{Spatial error of the localised/global filter estimates.\label{e2_filter_err}}%
{\includegraphics[width=6cm,keepaspectratio]{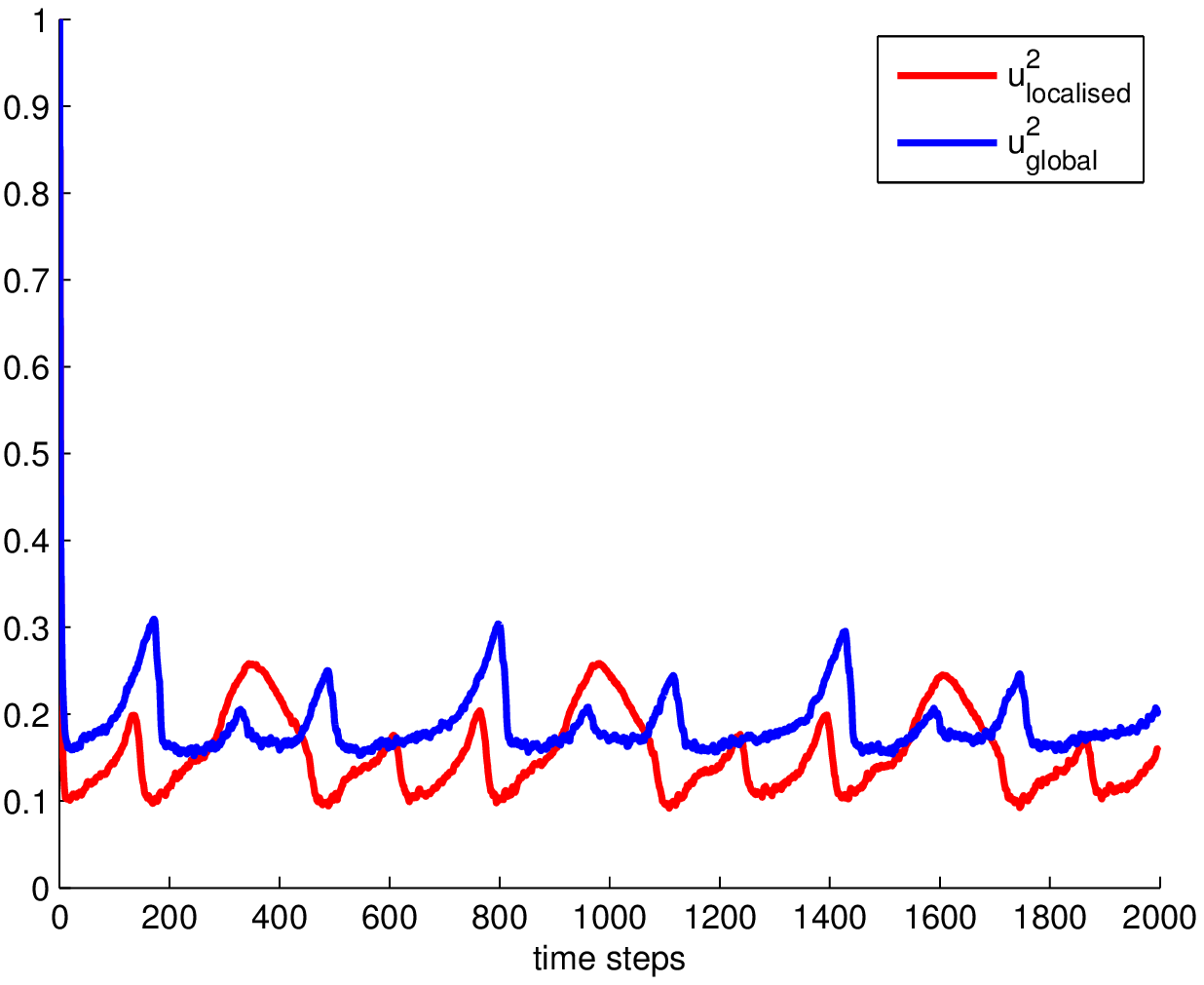}}
\end{figure}
\begin{figure}[htpb]
\subcaptionbox{The estimate of the localised filter, global filter and analytical solution computed at the point $x=1.4, y=1.4$ plotted over time steps $[135,185]$.\label{est_gl_loc}}%
{\includegraphics[width=6cm,keepaspectratio]{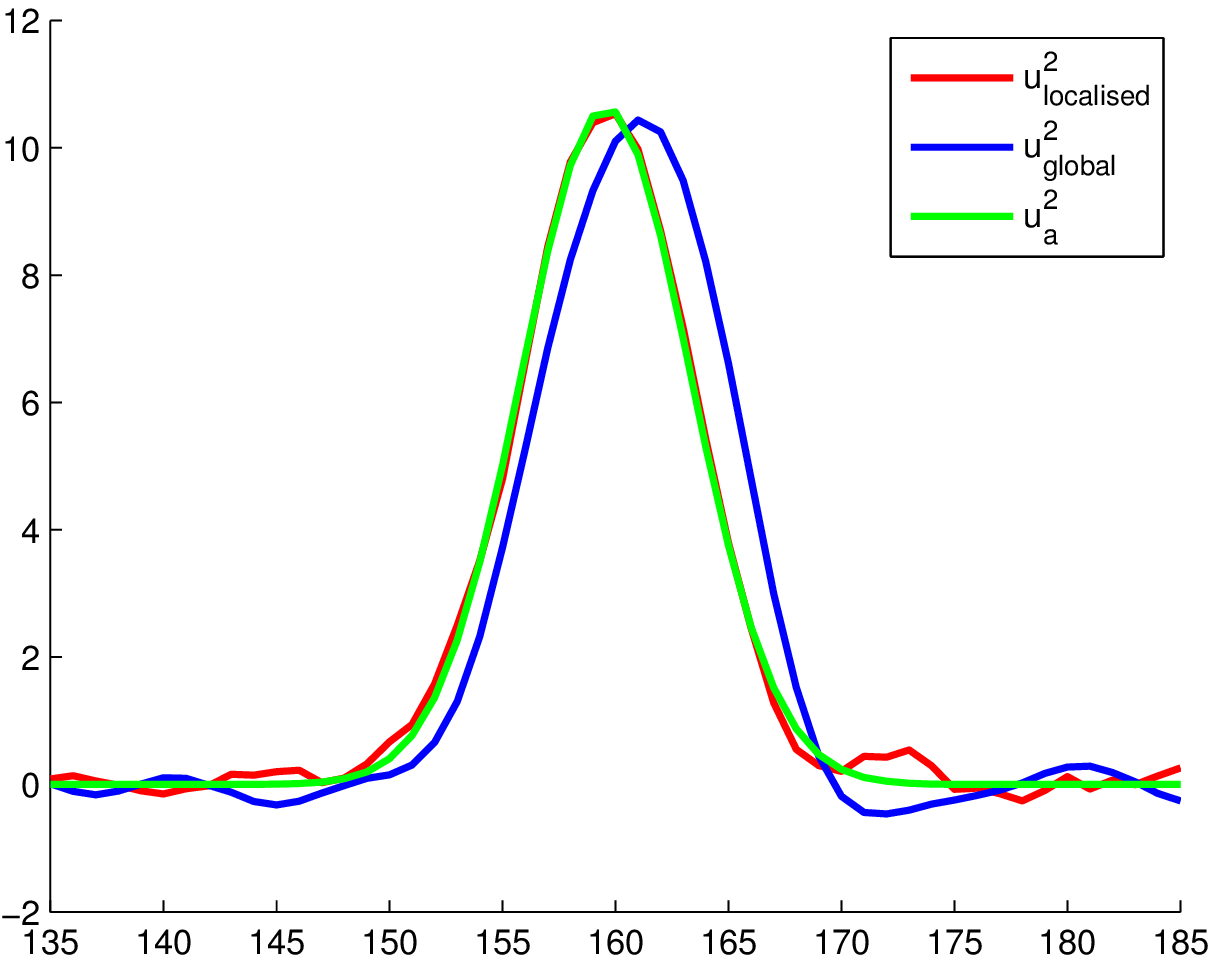}}\hfill
\subcaptionbox{Components of the localised and global Riccati operator corresponding to the point $x=1.4, y=1.4$ plotted over time steps $[0,300]$.\label{ric_gl_loc}}%
{\includegraphics[width=6cm,keepaspectratio]{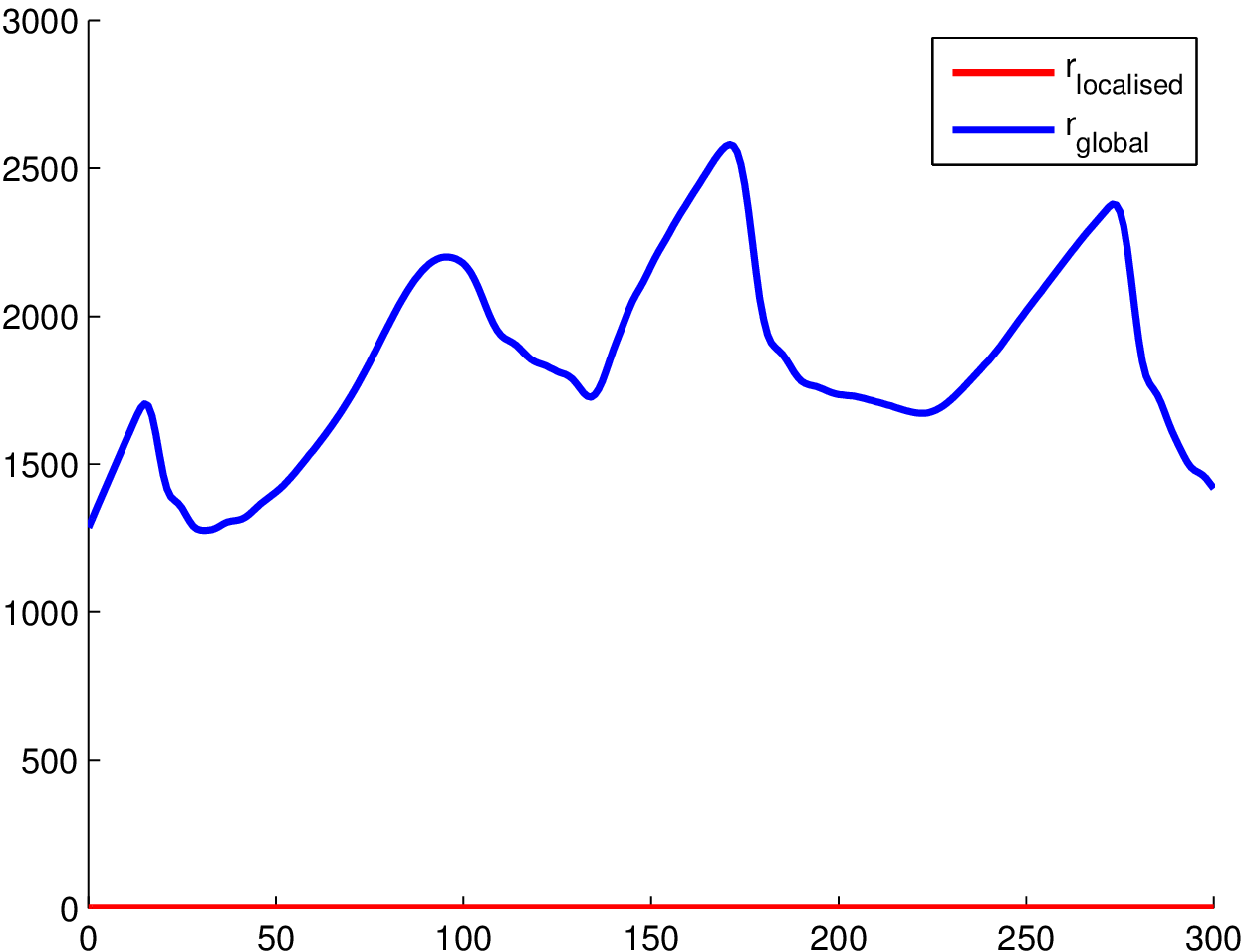}}

\subcaptionbox{The estimate, ellipsoid of the estimate and analytical solution computed at the point $x=1.4, y=1.4$ plotted over time steps $[135,185]$.\label{est_elips}}{\includegraphics[width=6cm,keepaspectratio]{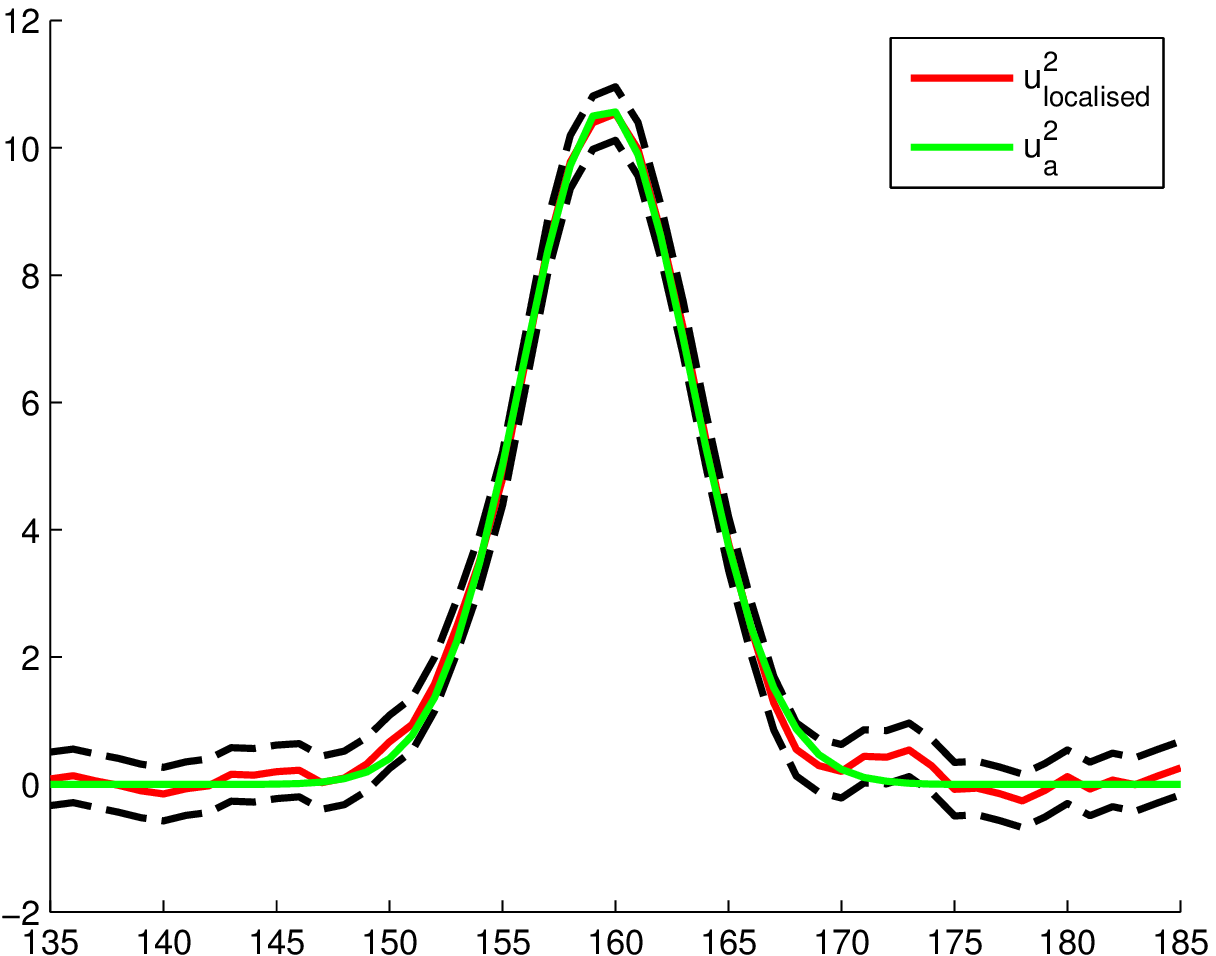}}\hfill
\subcaptionbox{Components of the Riccati operator corresponding to the point $x=1.4, y=1.4$ computed by localised filter with reinitialisation intervals 0.1 and 1 plotted over time steps $[0,300]$.\label{ric_reinit}}{
\includegraphics[width=6cm,keepaspectratio]{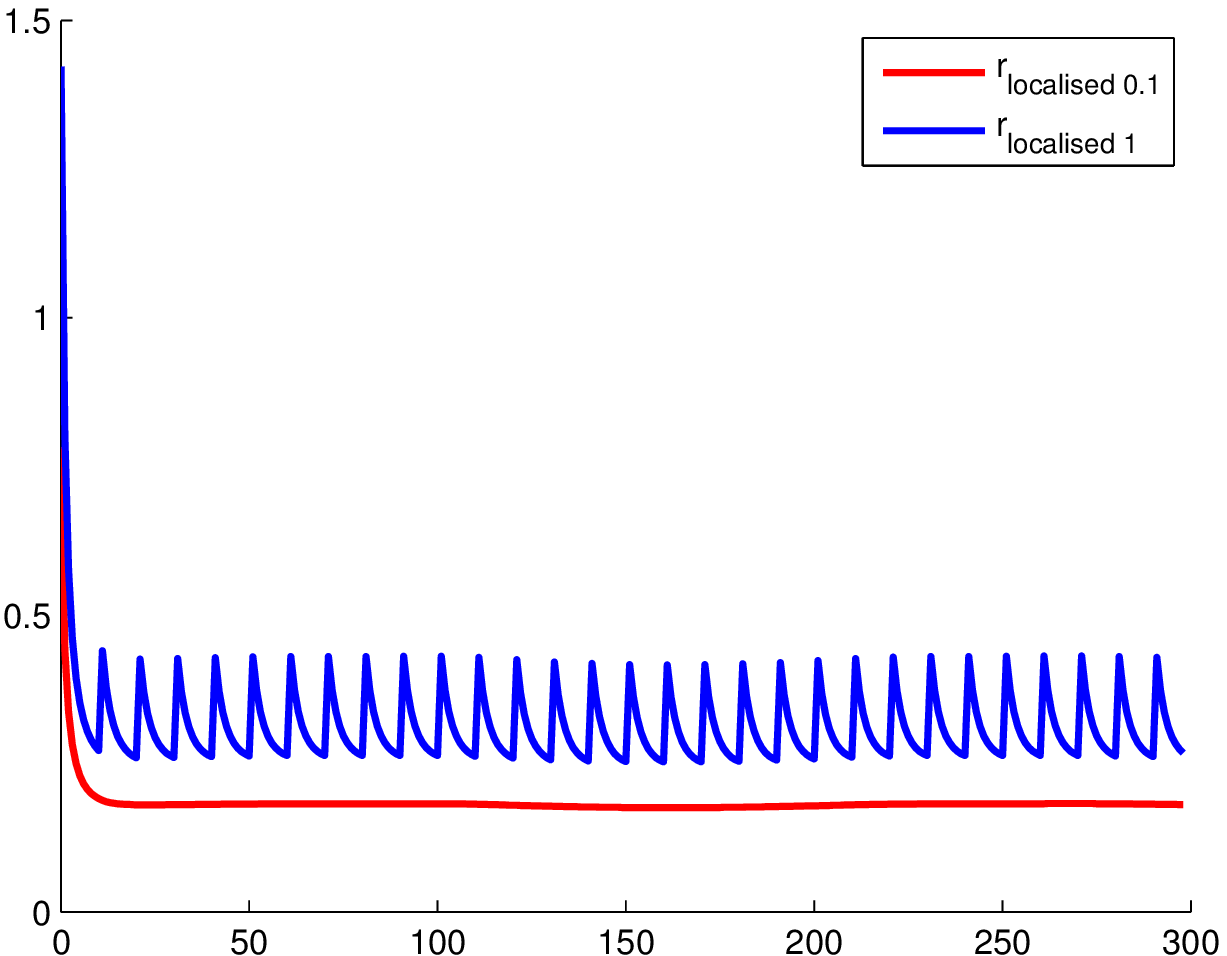}}

\subcaptionbox{Components of the Riccati operator corresponding to the point $x=1.4, y=1.4$ computed by localised filter with 225 and 900 FEM elements per subdomain plotted over time steps $[0,300]$. \label{e2_ricFEMInc}}{\includegraphics[width=6cm,keepaspectratio]{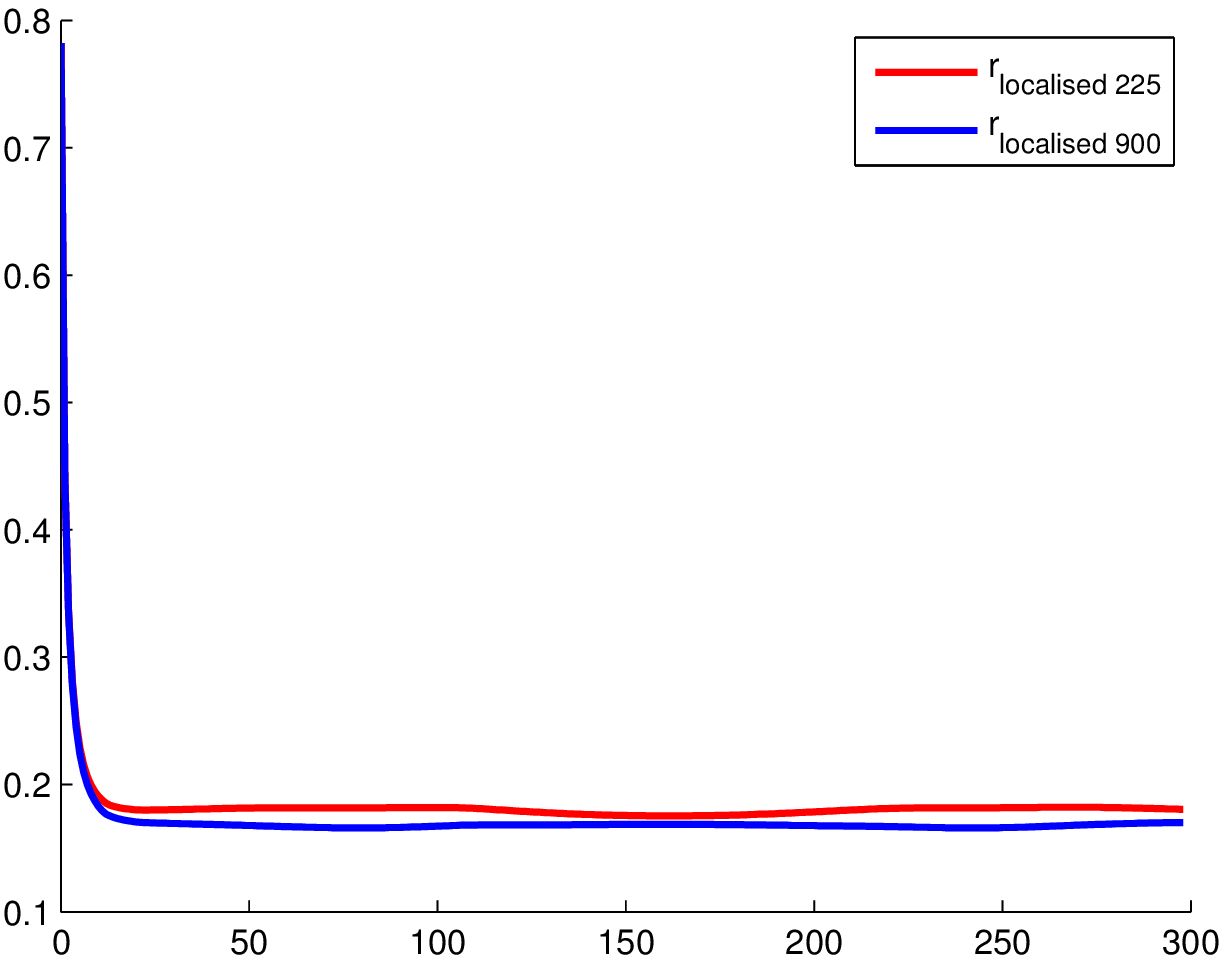}}\hfill
\subcaptionbox{Computational time taken for problems with different number of subdomains.\label{glinScale}}{
\includegraphics[width=6cm,keepaspectratio]{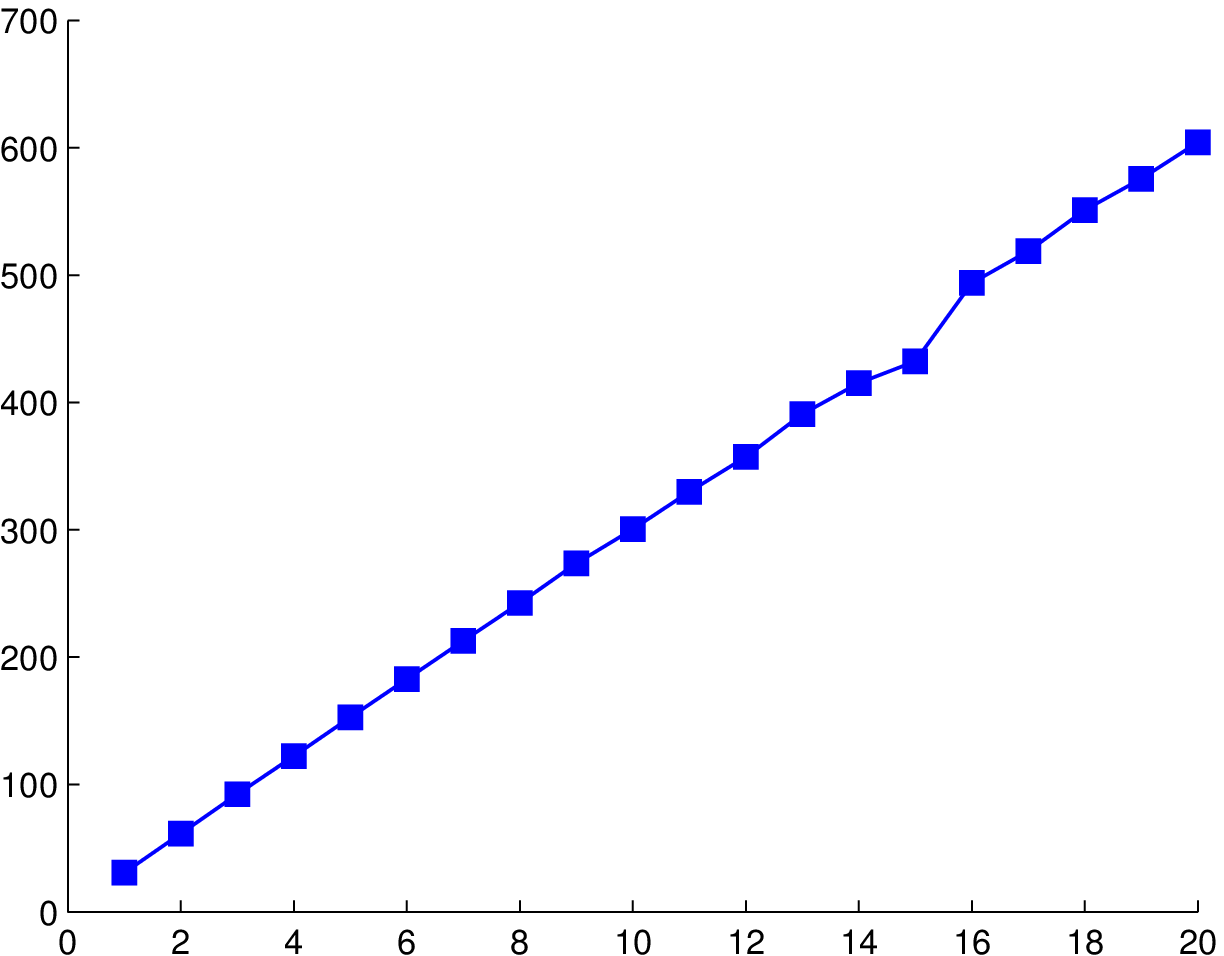}}

\end{figure}

\textbf{FEM discretization.} In this experiment a two dimensional rectangular domain $[0,3]\times[0,3]$ has been discretized by $2025$ bilinear finite elements. DD is applied by decomposing the domain into $9$ equal size subdomains $[0,1]\times[0,1]$ each over the x and the y-axis and discretized by $225$ finite elements. The underlying flow field $\mu$ is defined by time dependent harmonic functions:
\begin{equation}
\begin{array}{l}
    \mu_x(t,x,y) = sin(\pi - t/10) * 0.12 \\
    \mu_y(t,x,y) = sin(\pi / 2 - t/5) * 0.24
\end{array}
\end{equation}
The timestep is taken to be $\Delta t=0.1$ and the length of the simulation is set to be $2000$ time steps allowing for three full loops as suggested in Figure~\ref{e2_conf} (one loop requires $630$ time steps).

\textbf{Observations.} As in the first experiment, the analytical solution $u^2_a(x,y,t)$ is defined in the form of the Gaussian function~\eqref{GaussianFunc} with the following parameters:
\begin{equation}
\begin{array}{l}
\sigma = 0.1 + 0.01t, \; x_0 = 0.25, \; y_0 = 1.5 \\
    m_x(t,x,y) = (1 + cos(t / 10 - \pi)) * 1.2 \\
    m_y(t,x,y) = cos(t / 5 - \pi / 2) * 1.2
\end{array}
\end{equation}
The observations are generated by restricting the function $u^2_a(x,y,t)$ onto the nodes in subdomains $\Omega_i, i\in I_{obs}=\{3,4\}$. The structure of the observation matrix $\bm C_i$ is similar to the one from the first experiment, so the product $\bm C_i \bm M_i$ is diagonal with components equal $1$ if the corresponding FEM node is observed and $0$ otherwise. As above, the observation noise is taken to be uniformly distributed over the interval $[-0.5;0.5]$. The sensor's locations together with the sketch of the spill's trajectory are shown in Figure~\ref{e2_conf}.

\textbf{Uncertainty description.} Parameters of the localised filter at $i$-th subdomain are chosen as follows: $q=5$, $q_0=1.4$, $r=12$ and $\bm{Q}_i=qI$, $\bm{Q}_{0,i}=q_0 I$, $\bm{R}_i=rI$ and $\gamma_{T,i}=1.1$ describing a moderate level of trust in the FEM model over the subdomain $\Omega_i$, low confidence in the initial condition for the filter and a high trust to the observations. Figure~\ref{est_elips} shows the estimated value at the spatial point $x=1.4$, $y=1.4$ and demonstrates that the ground-truth is contained inside the ellipsoid. Examples of the observed fields are shown in Figures \ref{e2_obs25} and \ref{e2_obs180}, and the corresponding estimates generated by the localised filters are shown in Figures \ref{e2_est25} and \ref{e2_est180}. These figures show that even though the spill is not fully observed by the sensors, the local filters manage to reconstruct it with a reasonable precision level.

The performance of the localised estimate $u^2_{\text{localised}}$ is compared against the estimate $u^2_{\text{global}}$ of the so-called global filter which has been obtained by approximating the original $L^\infty$-ellipsoid by the $L^2$-ellipsoid i.e. filter without decomposition and reinitialisation. To compute the global filter equations~\eqref{eq:filteri_FEM}  are used with $\Omega_i=\Omega$ and the ellipsoids' matrices $\bm{Q},\bm{Q}_0$ and $\bm{R}$ defined as follows: to maintain consistency between descriptions \eqref{modelContEllipsoid1}, \eqref{modelContEllipsoid:i} of the global and local model errors respectively, and \eqref{observContEllipsoid1}-\eqref{observContEllipsoid:i} of the observation errors set $\bm{Q}=qI$, $\bm{Q}_{0}=q_0I$, $\bm{R}=rI$ and $\gamma_{T}=(T+1)N=1809$ where factor $N=9$ reflects the fact that $\ell(\Omega)=9\ell(\Omega_i)$ and time interval is set to be $T=200$. Figure~\ref{e2_filter_err} presents the spatial errors of the localised filters and the global filter. As one would expect, because of the nonstationary (in time) periodic behaviour of the underlying velocity field $\mu$, there are intervals where the errors are decreasing and increasing. At the same time, it is concluded, that in general both errors are not increasing over time and obey periodic behaviour. The respective estimation errors are in favour of the localised filter: $e_e(u^2_{\text{localised}})=16\%$ and $e_e(u^2_{\text{global}})=19\%$.\\

Figures~\ref{est_gl_loc}-\ref{ric_gl_loc} also suggests that even though the estimates are close to each other, the global filter overestimates the uncertainty in the system. Indeed, the uncertainty overestimation is demonstrated in the Figure~\ref{ric_gl_loc} where diagonal components of the Riccati operator $\bm{P}_i$ and $\bm{P}$ are plotted. Those components are computed at the spatial point $x=1.4, y=1.4$ and represent the uncertainty estimate provided by each of the filters via \eqref{eq:ME_i_FEM}. It can be seen that the localised filter's ellipsoid is much tighter than that of the global filter.

A comparison analysis of  the impact of the reinitialisation procedure onto the estimation error is in  figure~\ref{ric_reinit}.  The components of $\bm{P}_i$ obtained from the localised filter with the reinitialisation interval $\varepsilon$ equal to the time step of numerical integration $\Delta t=0.1$ are compared against the same components of $\bm{P}_i$, corresponding to the reinitialization interval of length $\varepsilon=1$. It is shown that the decrease of the reinitialisation interval leads to the decrease of the Riccati components which, in turn, reduces the estimation error.

Finally, components of $\bm{P}_i\bm{M}_i$ corresponding to the point $x=1.4, y=1.4$ computed with different FEM resolutions: 225 elements and 900 elements per subdomain are depicted in Figure~\ref{e2_ricFEMInc}. As it was expected, the increase of FEM degrees of freedom, does not increase the components of  $\bm{P}_i\bm{M}_i$ and the corresponding pointwise estimation error.

\subsection{Computational Performance}
\begin{figure}[htpb]
\centering
\includegraphics[width=10cm]{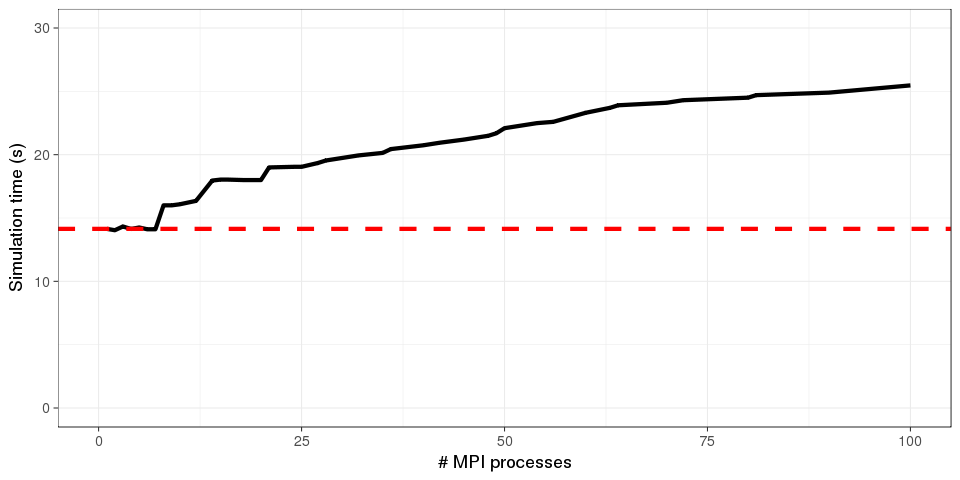}
  \caption{Simulation time to compute a $1000$ timestep solution plotted against the number of sub-domains. Scaling represent a weak scaling analysis where the computational size of the domain was increased in line with the number of MPI processes (i.e. the number of sub-domains equalled the number of MPI processes). The black line represents the total simulation time while dashed red line represents ideal scaling.}
\label{gfixedSubScale}
\end{figure}
Assume that  the global domain is decomposed into $N$ subdomains, each of them containing $N_{nd}^i$ finite elements. At each subdomain, the computational complexity of the localised filter is the combination of the computational complexity of the equations~\eqref{eq:feedbackMP} and~\eqref{eq:riccatiMP}.

To solve~\eqref{eq:feedbackMP} one needs to invert a matrix of size $N_{nd}^i \times N_{nd}^i$ which requires $\mathcal{O}((N_{nd}^i)^3)$ arithmetic operations. Similarly, to solve~\eqref{eq:riccatiMP} one needs $12\mathcal{O}((N_{nd}^i)^3)$. From these estimates, it is easy to conclude that: an increase of the number of finite elements corresponds to a dramatic increase in computational costs; solving the Riccati equation costs approximately $12$ times more  then solving the filter equation.

If in the above analysis, $N^i_{nd}$ is replaced by the total amount of FEM grid nodes $N_{nd}$, the complexity estimate of the global filter becomes $c_g=13\mathcal{O}((N_{nd})^3)$. The rough approximation of $N_{nd}$ by the $N N^i_{nd}$ results in
\begin{equation}
\label{complexGlobal}
c_g(N)=13N^3\mathcal{O}((N^i_{nd})^3).
\end{equation}
For the computational complexity estimation of the localised filter one also needs to keep into account the Schwartz iterations. Assume, that $p$ iterations were performed, the total number of operations for equation~\eqref{eq:feedbackMP} becomes $pN\mathcal{O}((N_{nd}^i)^3)$, and $12N\mathcal{O}((N_{nd}^i)^3)$ for the equation~\eqref{eq:riccatiMP}. Therefore, the total amount of arithmetic operations for the algorithm of the localised minimax filter for one time step is estimated as
\begin{equation}
\label{complexLocal}
c_l(N)=(p+12)N\mathcal{O}((N_{nd}^i)^3)
\end{equation}
Since $(p+12)N \ll 13N^3$, it is clear from~\eqref{complexGlobal}-\eqref{complexLocal} that localised filter provides significant complexity reduction comparing to the traditional global filter.

Finally, a very basic scaling benchmark has been performed on an IBM NextScale nx360 compute server. Each node consists of two 18-core Intel Xeon Processor E5-2699 v3 (2.3 GHz, 45 MB L3 cache per processor), 36 core total, forming a single NUMA (Non-Uniform Memory Architecture) unit with 256 GB of RAM and 10 GbE Infiniband network interconnect. Simulations investigated computational performance when increasing the number of sub-domains at the same rate as number of MPI processes. All simulations considered a 16 x 16 element sub-domain with number of sub-domains increased from 1 – 100 (and consequently MPI processes). MPI overheads were a result of 1) neighbour-to-neigbour data exchange of boundary data to propagate solution between sub-domains and 2) a global MPI reduction to compute the difference in the computed solution across sub-domain boundaries for convergence of the Schwarz solver. The solution was deemed to converge when this error was less than some predefined threshold. Computation of the error required a global MPI reduction operation at each iteration of the Schwarz solver to define convergence. The MPI synchronisation introduced at each time step incurs a latency and communication penalty; it also exacerbates any potential load imbalances as computation is constrained to the slowest process.

Figure~\ref{gfixedSubScale} presents the simulation time when running up to 100 MPI processes distributed across 5 nodes (with maximum of 20 MPI processes on any individual node). These results present a weak scaling configuration where problem size is increased together with number of computational cores (i.e. for each increase in number of cores, number of sub-domains, of fixed size, are increased correspondingly). An ideal model would produce no increase in simulation time, as workload assigned to each core remains fixed. The reality is that MPI synchronisation, along with contention of processes when more than one process is deployed on a single node will lead to performance overheads. Figure~\ref{gfixedSubScale} demonstrates that deploying on up to 5 cores produces no change in simulation time. This largely results from the fact that processes are equally distributed across nodes so that when running 5 MPI processes there is a single process on each node thereby leading to no contention issues. Beyond this there is some increase in model simulation time, potentially due to contention of MPI processes for resources. Modern multicore systems are designed to allow cluster of cores to share certain hardware components such as cache, memory controllers and interconnects. Hence MPI processes running on the same node may compete for the same resources and consequently suffer from performance degradation. The approximately linear increase in simulation time suggests that the performance degradation in this simulation is a result of 1) slowdown due to contention for hardware resources and 2) MPI overheads primarily due to the global communication required for error computation. The MPI overhead due to neighbour-to-neighbour data exchange required for the Schwarz synchronization is a local communication only which is not expected to increase computation cost beyond a five point stencil implementation (i.e. one neighbour in each direction). Despite performance overheads from MPI synchronisation and resource contention, these results demonstrate the benefit of deploying the model in a sub-domain parallel approach, providing an increase in domain size of 100,000 with an increase in total simulation time of 69\% when deploying across 100 cores.

\section{Concluding remarks}
\label{sec:Conclusion}
In this work, a new state estimation algorithm is proposed for advection dominated flows with deterministic/stochastic (non-Gaussian) uncertainty description of $L^\infty$-type. The algorithm is recursive, i.e. the current estimate depends on the previous one and on the current observation, computationally efficient and scalable. It delivers both integral and pointwise estimates which converge to the corresponding continous quantities over each local subdomain.
\appendix
\section{FEM approximations}
This appendix expands on the FEM approximations of the continuous local filtering subproblem~\eqref{eq:obs:i},\eqref{LocalIterProblemi}\eqref{eq:obsnoise:i}.

\subsection{FEM model for $i$-th local subproblem}
\label{sec:appendix_FEM}
To simplify the presentation consider the case of two spatial dimensions, $n=2$. To apply FEM~\eqref{eq:state} is reformulated in the weak form. Specifically, $u\in L^\infty(0,T,H^1_0(\Omega))$ is the unique solution of~\eqref{eq:state} if for any $v \in H^1(\Omega)$ the following integral equality holds true:
\begin{equation} \label{weakPDEnoBI}
\int_{\Omega} \dot u v d\Omega = \int_{\Omega}\epsilon \Delta u vd\Omega - \int_{\Omega}(\mu_1\partial_x u + \mu_2 \partial_y u) v d\Omega\,.
\end{equation}
The divergence theorem is applied in order to enforce the boundary conditions (in the weak sense):
\begin{equation}
\label{weakPDE}
  \begin{split}
\int_{\Omega} \dot u v d\Omega =
& -\int_{\Omega}\epsilon (\partial_x u \partial_x v + \partial_y u \partial_y v) d\Omega \\
& + \int_{\Omega}u(\mu_1\partial_x v + \mu_2\partial_y v) d \Omega \\
& + \int_{\partial \Omega} \epsilon \frac{\partial u}{\partial \bm n} v d\Omega - \int_{\partial \Omega} (\mu_1 + \mu_2)g v d\Omega
  \end{split}
\end{equation}
where the function $g$ defines Dirichlet data. Similarly, for $i$-th local subdomain at $n$-th Schwartz iteration the weak formulation of~\eqref{LocalIterProblemi} takes the following form:
\begin{align}
\int_{\Omega_i} \dot u_i^{n+1} v d\Omega = \nonumber  \\
& -\int_{\Omega_i}\epsilon (\partial_x u_i^{n+1} \partial_x v + \partial_y u_i^{n+1} \partial_y v) d\Omega  \label{DiffInt} \\
& + \int_{\Omega_i}u_i^{n+1}(\mu_1\partial_x v + \mu_2\partial_y v) d \Omega_i \label{AdvInt} \\
& + \int_{\partial \Omega \cap \partial \Omega_i} \epsilon \frac{\partial u_i^{n+1}}{\partial n} v d\Gamma \label{NeumGlobBoundInt} \\
& + \int_{\Gamma_{i,j}^{in}\in \Gamma_i^{in}} \epsilon \frac{\partial u_i^{n+1}}{\partial n} v d\Gamma \label{NeumInBoundInt} \\
& + \int_{\Gamma_{i,j}^{out}\in \Gamma_i^{out}} \epsilon \frac{\partial u_i^{n+1}}{\partial n} v d\Gamma \label{NeumOutBoundInt} \\
& - \int_{\partial \Omega \cap \partial \Omega_i} (\mu_1 + \mu_2)u_i^{n+1} v d\Gamma  \label{DirGlobBoundInt} \\
& - \int_{\Gamma_{i,j}^{in}\in \Gamma_i^{in}} (\mu_1 + \mu_2)u_j^{n} v d\Gamma  \label{DirInBoundInt} \\
& - \int_{\Gamma_{i,j}^{out}\in \Gamma_i^{out}} (\mu_1 + \mu_2)u_i^{n+1} v d\Gamma \label{DirOutBoundInt}\,,
\end{align}
where $d\Omega$ and $d\Gamma$ denote the differentials for the integrals over the subdomains and parts of their boiundaries. In d-ADN the decomposition integral \eqref{NeumOutBoundInt} vanishes.

The FEM discretization of~\eqref{LocalIterProblemi} proceeds by means of polygonal finite elements $\Lambda_m$, $m=1,..,N_i^{el}$, i.e., the domain $\Omega_i$ is divided into a finite number of polygones $\Lambda_m$ with vertices $x_k$, $k=1..N_{nd}^i$, $\Omega\approx\cup \Lambda_m$. The vertices $x_k$ form the FEM grid, and at each node $x_s$ of this grid, the corresponding basis function $\phi_k$ satisfies
\begin{equation}
\label{kronrule}
\phi_k(x_s)=\delta_{ks}\,,\quad \delta_{ks}\text{ is the Kronecker delta. }
\end{equation}
The most simple basis functions are tensor products of 1D piece-vise linear functions or so-called ``hat functions''. In what follows,  the following subsets of indices are adopted that define the subsets of inflow/outflow boundary nodes as
\begin{equation}
  \begin{split}
    D_{in/out}&= \{k:x_k \in \Gamma_{i,j}^{in/out}\}\,,\\
    N_{in}&= \{k:\exists m \text{ that } \partial \Lambda_m \cap \Gamma_{i,j}^{in} \neq \emptyset \text{ and } x_k \in \Lambda_m\}\,,\\
I&=\{s:x_s\not\in D_{in/out}\cup N_{in}\}\,.
  \end{split}
\end{equation}
Now, $u_i^{n+1}$ is approximated as follows: \[
u_i^{n+1}=\sum_{k=1}^{N_{nd}^i} u^{n+1}_{ik}(t)\phi_k\,.
\]
To find the coefficients $u_{ik}^{n+1}$ the above representation is substituted into~\eqref{DiffInt}-\eqref{DirOutBoundInt} which leads to the FEM model for the coefficients:
\begin{equation}
\label{FEMDProblems}
\left\{ \begin{array}{l}
    \bm{M}_i\frac{d\bm{u}_i^{n+1}}{dt}=
    \bm{S}_i(t)\bm{u}_i^{n+1} + \bm{f}_i(t;\bm{u}_j^n) + \bm{M}_i\bm{e}_i\\
    \bm{u}_i^{n+1}(0)=\bm{u}_i^0+\bm{e}_{0,i}
  \end{array}\right.
\end{equation}
where $\bm{u}_i^{n+1}=(u_{i1}^{n+1}(t),..,u_{iN_{nd}^i}^{n+1}(t))^T$ is the vector of FEM coefficents representing the FEM approximation of $u_i^{n+1}$, $\bm{u}_i^0$ is the FEM approximation of the restriction of $u_0$ onto $\Omega_i$, $\bm{e}_i$ and $\bm{e}_{0,i}$ are the vectors of coefficients of the spatial FEM discretization of the model and initial errors, $\bm{M}_i=\{\int_{\Omega_i}\phi_k \phi_sd\Omega\}_{k,s=1}^{N_{nd}^i}$ is the local mass matrix, $\bm{S}_i$ is the local stiffness matrix defined by
\begin{equation}
\label{locStiff}
	\bm{S}_i(t) = S_i^\Omega(t) + S_i^{N_{in}}(t)+ S_i^{D_{out}}(t)
\end{equation}
where \[
	S_i^\Omega(t)  =
	\begin{pmatrix}
	  S^\Omega_{D_{in} D_{in}} & S^\Omega_{D_{out} D_{in}} & S^\Omega_{N_{in}/\Gamma D_{in}} & S^\Omega_{I D_{in}}\\
	  S^\Omega_{D_{in} D_{out}} & S^\Omega_{D_{out} D_{out}} & S^\Omega_{N_{in}/\Gamma D_{out}} & S^\Omega_{I D_{out}}\\
	  S^\Omega_{D_{in} N_{in}/\Gamma} & S^\Omega_{D_{out} N_{in}/\Gamma} & S^\Omega_{N_{in}/\Gamma N_{in}/\Gamma} & S^\Omega_{I N_{in}/\Gamma}\\
	  S^\Omega_{D_{in} I} & S^\Omega_{D_{out} I} & S^\Omega_{N_{in}/\Gamma I} & S^\Omega_{II}
	\end{pmatrix}
\]
and \[
\begin{split}
	 S_{X,Y}^\Omega:=\{&-\int_{\Omega_i}\epsilon (\partial_x \phi_k \partial_x \phi_s + \partial_y\phi_k \partial_y \phi_s) d\Omega + \int_{\Omega_i}\phi_k(\mu_1\partial_x \phi_s + \mu_2\partial_y \phi_s) d \Omega\\
	 &+ \int_{\partial \Omega \cap \partial \Omega_i} \epsilon \phi_s\frac{\partial \phi_k}{\partial n} d\Gamma\}_{k\in X,s\in Y}\,,
\end{split}
\]
with $X$ and $Y$ corresponding to the subsets of the indices of the basis functions, e.g., $X=D_{in}$ and $Y=D_{out}$. In fact, $S_i^\Omega$ absorbs the integrals~\eqref{DiffInt}, \eqref{AdvInt} and \eqref{NeumGlobBoundInt}. Now, $S_i^{N_{in}}(t)$ is defined as follows: \[
S_i^{N_{in}}(t):=
\begin{pmatrix}
	 S^{N_{in}}_{D_{in} D_{in}} & 0 & S^{N_{in}}_{N_{in}/\Gamma D_{in}} & 0 \\
 	 0 & 0 & 0 & 0 \\
 	 S^{N_{in}}_{D_{in} N_{in}/\Gamma} & 0 & S^{N_{in}}_{N_{in}/\Gamma N_{in}/\Gamma} & 0 \\
 	 0 & 0 & 0 & 0
\end{pmatrix}
\]
with $S^{N_{in}}_{X,Y} = \{\int_{\Gamma_{i,j}^{in}\in \Gamma_i^{in}} \epsilon \phi_s\frac{\partial \phi_k}{\partial n} d\Gamma\}_{k\in X,s\in Y}$. Clearly, $S_i^{N_{in}}(t)$ absorbs~\eqref{NeumInBoundInt}. Finally,  $S_i^{D_{out}}(t)$ is given by \[
S_i^{D_{out}}(t):=\begin{pmatrix}
	 0 & 0 & 0 & 0 \\
 	 0 & -S^{D_{out}} & 0 & 0 \\
 	 0 & 0 & 0 & 0 \\
 	 0 & 0 & 0 & 0
	\end{pmatrix}
\] with $S^{D_{out}}=\{\int_{\Gamma_{i,j}^{out}\in \Gamma_i^{out}} (\mu_1 + \mu_2)\phi_k \phi_s d\Gamma \}_{k,s=1}^{N_{nd}^i}$, so that $S_i^{D_{out}}$ absorbs~\eqref{DirOutBoundInt}. The local source vector absorbs the integrals \eqref{DirInBoundInt} and \eqref{DirGlobBoundInt} (the latter equals $0$ as the global problem has homogeneous Dirichlet boundary condition). It is defined by
\begin{equation}
\label{locSource}
	\bm{f}_i(t;\bm{u}_j^n) = \bm{M}_i\tilde{f}_i(t) + [S_{i}^{D_{in}}(t) \bm{u}_j^{n,D_{out}}, 0_{D_{out}},0_{N_{in}/\Gamma},0_{I}]^T
\end{equation}
where $\tilde{f}_i(t)=(f_i(x_{i1})\dots f_i(x_{iN^i_{nd}}))^\top$ is the FEM approximation of the restriction of the source term $f$ onto $\Omega_i$,
$S_{i}^{D_{in}}$ is defined by substituting 'out' by 'in' in the definition of  $S^{D_{out}}$, $\bm{u}_j^{n,D_{out}}$ denotes the sub-vectors of $\bm{u}_j^{n}$ with components $\bm{u}_j^{n}(s)$ such that $s\in D_{in}^i\cap D_{out}^j$ (here $D^i_{in}$ denotes $D_{in}$ of $\Omega_i$). $\bm{u}_i^{D_{out}}$ is defined analogously.

Finally,  note that the block-structure of the stiffness matrix suggests the following splitting of the vector $\bm{u}_i^{n+1}$:
\begin{equation}
\bm{u}_i^{n+1}=[\bm{u}_i^{D_{in}}, \bm{u}_i^{D_{out}}, \bm{u}_i^{N_{in}/\Gamma}, \bm{u}_i^I]^T
\end{equation}
and $\bm{u}_i^{N_{in}}=[\bm{u}_i^{D_{in}}, \bm{u}_i^{N_{in}/\Gamma}]^T$.

\section{Proof of Lemma~\ref{l:MF_FEM}}
\label{sec:proofs}
\begin{proof}
Recall from section~\ref{sec:local-minimax-filter} that $w_i^{n+1}$ depends linearly on $q_j$ and $ w_j^n$, hence the minimax estimate of $w_i^{n+1}$ is given by $\hat w_i^{n+1}$, the solution of~\eqref{eq:wi} with the Dirichlet boundary condition $w_i^{n+1} = \hat u_j^n$ on $\Gamma^{in}_{i,j}\in \Gamma_i^{in}$, where $\hat u_j^n$ denotes the $(n,j)$-filter obtained on the $n$-th iteration of the Schwartz iterative procedure. Analogously to~\eqref{FEMDProblems}, the FEM model of~\eqref{eq:wi} with the Dirichlet boundary condition $w_i^{n+1} = \hat u_j^n$ on $\Gamma^{in}_{i,j}\in \Gamma_i^{in}$ is introduced:
\begin{equation}
\label{eq:wi_FEM}
\left\{ \begin{array}{l}
    \bm{M}_i\frac{d\bm{\hat w}_i^{n+1}}{dt}=
    \bm{S}_i(t)\bm{\hat w}_i^{n+1} + \bm{f}_i(t;\bm{\hat u}_j^n)\\
    \bm{\hat w}_i^{n+1}(0)=\bm{u}_i^0
  \end{array}\right.
\end{equation}
To compute $\bm{\hat q}_i^{n+1}$ \eqref{eq:MF_qi} is discretised:
\[
\begin{split}
l(\hat q_i^{n+1}(T)) &=\gamma_{T,i}^{-1}\int_{\Omega_i\times (0,T)} r_i(t,x) (H_i p_i)(t,x)\tilde y_i(t,x) dxdt  \\
&= \gamma_{T,i}^{-1}\int_{\Omega_i\times (0,T)} p_i(t,x) (H_i^\star r^{\frac12}_i r^{\frac12}_i \tilde y_i)(t,x) dxdt \\
&\approx  \gamma_{T,i}^{-1}\int_{\Omega_i\times (0,T)} \sum_s^{N_{nd}^i} p_i(t,x_s) \phi_s(x) \\
& \times \int_{\Omega_i} \sum\limits_{n,m,j}^{N_{nd}^i}h(x_n-z_m) r^{\frac12}_i(z_m)\phi_n(x) \phi_m(z) r_i^{\frac12}(z_j) \tilde y_i(z_j,t)\phi_j(z)dzdx\\
&= \gamma_{T,i}^{-1}\int_{\Omega_i\times (0,T)} \sum_s^{N_{nd}^i} p_i(t,x_s) \phi_s(x) \phi_n(x) \\
& \times \sum\limits_{n,m,j}^{N_{nd}^i}h(x_n-z_m) r^{\frac12}_i(z_m) \int_{\Omega_i}\phi_m(z) \phi_j(z)dz\,r_i^{\frac12}(z_j) \tilde y_i(z_j,t) dx\\
&=\gamma_{T,i}^{-1}\int_0^T(\bm{p}_i(t), \bm{M}_i \bm{C}_i \bm{R}_{i}^{\frac 12} \bm{M}_i\bm{R}_i^{\frac 12} \bm{\tilde y}_i(t))dt
\end{split}
\]
where $\phi_s$ and $\{x_s\}_{s=1}^{{N_{nd}^i}}$ are defined as in section~\ref{sec:appendix_FEM}, and $\tilde y_i = y_i - H_i \hat w_i^{n+1}$,
\[
\bm{\tilde y}_i = (\tilde y_i(x_1,t),\dots,\tilde y_i(x_{N_{nd}^i}))^\top\,,
\]
FEM approximation of~\eqref{eq:pzi} reads as follows\footnote{For instance, the term $\bm{Q}_i^{\frac12}\bm{M}_i\bm{Q}_i^{\frac12}\bm{z}_i$ represents the matrix resulting from the FEM approximation of the integral $\int_{\Omega_i} \phi_k q^2_i z_i dx$, e.g.: $\int_{\Omega_i} \phi_k q^2_i z_i dx =\int_{\Omega_i} (\phi_k q_i) (q_i z_i) dx = \int_{\Omega_i} (\sum_n \phi_k(x_n)q_i(x_n) \phi_n(x)) (\sum_j q_i(x_j)z_i(x_j) \phi_j dx = \int_{\Omega_i} q_i(x_k) \phi_k(x) \sum_j q_i(x_j)z_i(x_j) \phi_j(x) dx$ as $\phi_k(x_n)=\delta_{kn}$.
}:
\begin{equation}\label{eq:pzi_FEM}
\left\{ \begin{array}{l}
    \bm{M}_i \dot{\bm{z}}_{i} = -\bm{S}^\top_i \bm{z}_i + \gamma^{-1}_{T,i}\bm{M}_i\bm{C}_i^\top\bm{R}_i^{\frac 12}\bm{M}_i\bm{R}_i^{\frac 12}\bm{C}_i\bm{M}_i\bm{p}_i \\
    \bm{M}_i \bm{z}_i(T) = \bm{M}_i\bm{l}_i\\
    \bm{M}_i \dot{\bm{p}}_{i} = \bm{S}_i \bm{p}_i + \gamma_{T,i}\bm{Q}_i^{\frac12}\bm{M}_i\bm{Q}_i^{\frac12}\bm{z}_i \\
    \bm{M}_i\bm{p}_i(0) = \gamma_{T,i}\bm{Q}_{0,i}^{\frac 12}\bm{M}_i\bm{Q}_{0,i}^{\frac 12}\bm{z}_i(0)
\end{array}\right.
\end{equation}
Define $\bm{d}_i:=\bm{M}_i\bm{p}_{i} $ and multiply the first two equalities of~\eqref{eq:pzi_FEM} by $\bm{M}_i^{-1}$:
\begin{equation}
\left\{ \begin{array}{l}
    \dot{\bm{z}}_{i} = -\bm{M}_i^{-1}\bm{S}^\top_i \bm{z}_i + \gamma^{-1}_{T,i}\bm{C}_i^\top\bm{R}_i^{\frac 12}\bm{M}_i\bm{R}_i^{\frac 12}\bm{C}_i\bm{d}_i\\
    \bm{z}_i(T) = \bm{l}_i\\
    \dot{\bm{d}}_{i} = \bm{S}_i \bm{M}_i^{-1} \bm{d}_i + \gamma_{T,i}\bm{Q}_i^{\frac12}\bm{M}_i\bm{Q}_i^{\frac12}\bm{z}_i \\
    \bm{d}_i(0) = \gamma_{T,i}\bm{Q}_{0,i}^{\frac 12}\bm{M}_i\bm{Q}_{0,i}^{\frac 12}\bm{z}_i(0)
\end{array}\right.
\end{equation}
It is well known that the above Hamiltonian system for $\bm{d}_i$ and $\bm{z}_i$ has the unique solution for any $\bm{l}_i$. It is here claimed that $\bm{d}_i = \bm{P}_i\bm{z}_i$ where $\bm{P}_i$ solves the matrix DRE in~\eqref{eq:filteri_FEM}. Indeed, by substituting $\bm{d}_i' = \bm{P}_i\bm{z}_i$ into the differential equation for $\bm{d}_i$ it follows that: \[
\begin{split}
  &\bm{S}_i \bm{M}_i^{-1}\bm{P}_i\bm{z}_i + \gamma_{T,i}\bm{Q}_i^{\frac12}\bm{M}_i\bm{Q}_i^{\frac12}\bm{z}_i
=\dot{\bm{P}}_i\bm{z}_i - \bm{P}_i \bm{M}_i^{-1}\bm{S}^\top_i \bm{z}_i + \gamma_{T,i}^{-1}\bm{P}_i\bm{C}_i^\top\bm{R}_i^{\frac 12}\bm{M}_i\bm{R}_i^{\frac 12}\bm{C}_i\bm{P}_i\bm{z}_i
\end{split}
\]
Hence, $\bm{d}_i=\bm{d}_i'$ solves the aforementioned Hamiltonian system and coincides with its unique solution, $\bm{d}_i=\bm{d}_i'$. Now, the equation for $\bm{z}_i$ reads as follows: \[
\dot{\bm{z}}_{i} = -\bm{M}_i^{-1}\bm{S}^\top_i\bm{z}_i + \gamma^{-1}_{T,i}\bm{C}_i^\top\bm{R}_i^{\frac 12}\bm{M}_i\bm{R}_i^{\frac 12}\bm{C}_i\bm{P}_i\bm{z}_i\,, \bm{z}_i(T) = \bm{l}_i\,.
\]
Assume that $\bm{\tilde u}_i$ solves the first equation in ~\eqref{eq:filteri_FEM} provided $\bm{u}_i^0=0$, $\bm{f}_i(t;\bm{\hat u}_j^n)=0$ and $\bm{y}_i$ is substituted by $\bm{\tilde y}_i$. It is obtained:
\[
\begin{split}
l(\hat q_i^{n+1}(T)) &=\gamma^{-1}\int_{\Omega_i\times (0,T)} r_i(t,x) (H_i p_i)(t,x) \tilde y_i(t,x) dxdt  \\
&\approx\gamma^{-1}\int_0^T(\bm{p}_i(t), \bm{M}_i \bm{C}_i \bm{R}_{i}^{\frac 12} \bm{M}_i\bm{R}_i^{\frac 12} \bm{\tilde y}_i(t))dt \\
&=\gamma^{-1}\int_0^T(\bm{d}_i(t), \bm{C}_i \bm{R}_{i}^{\frac 12} \bm{M}_i\bm{R}_i^{\frac 12} \bm{\tilde y}_i(t))dt\\
&=\gamma^{-1}\int_0^T (\bm{z}_i(t),  \bm{P}_i\bm{C}_i \bm{R}_{i}^{\frac 12} \bm{M}_i\bm{R}_i^{\frac 12} \bm{\tilde y}_i(t))dt\\
&=\int_0^T (\bm{z}_i(t), \dfrac{d \bm{\tilde u}_i}{dt} - \bm{S}_i \bm{M}_i^{-1}\bm{\tilde u}_i + \gamma_{T,i}^{-1}\bm{P}_i \bm{C}_i^\top\bm{R}_i^{\frac 12}\bm{M}_i\bm{R}_i^{\frac 12}\bm{C}_i\bm{\tilde u}_i)dt\\
&=(\bm{l}_i, \bm{\tilde u}_i(T)) - \int_0^T (\dot{\bm{z}}_i, \bm{\tilde u}_i)dt \\
& - \int_0^T(\bm{z}_i,\bm{S}_i \bm{M}_i^{-1}\bm{\tilde u}_i - \gamma_{T,i}^{-1}\bm{P}_i \bm{C}_i^\top\bm{R}_i^{\frac 12}\bm{M}_i\bm{R}_i^{\frac 12}\bm{C}_i\bm{\tilde u}_i)dt\\
&=(\bm{l}_i, \bm{\tilde u}_i(T))
\end{split}
\]
Now, \eqref{eq:uwq} implies that $l_i(\hat u_i^{n+1}) = l_i(\hat w_i^{n+1}) + l_i(\hat q_i^{n+1})$ so that \[
l_i(\hat u_i^{n+1})\approx (\bm{l}_i, \bm{\tilde u}_i(T)+\bm{M}_i\bm{\hat w}_i^{n+1}(T))\,.
\]
Finally, it is straigntforward to check by differentiating that $\bm{\hat u}_i^{n+1}:=\bm{\tilde u}_i+\bm{M}_i\bm{\hat w}_i^{n+1}$.
\end{proof}



\end{document}